\documentclass{article}
\usepackage[utf8]{inputenc}
\usepackage{babel}
\usepackage[authoryear]{natbib}
\usepackage{amsfonts}
\usepackage{amssymb,amsmath}
\usepackage[T1]{fontenc}  
\usepackage{hhline}
\usepackage{calc}
\usepackage{dsfont}
\usepackage{color}
\usepackage{amsthm}
\usepackage{mathtools}
\usepackage{enumitem}
\usepackage[svgnames]{xcolor}
\usepackage{geometry}
\usepackage{algorithmic}
\usepackage{algorithm}
\usepackage{bbm}
\usepackage{xspace}
\definecolor{mongris}{gray}{0.75}
\definecolor{bleuclair}{rgb}{0.83,0.91,0.97}
\definecolor{bleufonce}{rgb}{0.2,0.59,0.85}
\definecolor{fond}{rgb}{0.99,0.96,0.9}
\definecolor{monviolet}{rgb}{0.64,0.41,0.74}
\definecolor{essai1}{rgb}{0.97,0.97,0.97}

\newcommand{\bb}{\mathbb}
\newcommand{\E}{\mathbb{E}}
\newcommand{\s}{\geq}
\newcommand{\m}{\leq}

\newcommand{\lra}{\longrightarrow}

\newcommand{\1}{\mathbbm{1}}

\newcommand{\fonction}[5]{\begin{array}{lrcl}
#1: & #2 & \longrightarrow & #3 \\
    & #4 & \longmapsto & #5 \end{array}}

\newcommand{\nset}{\mathbb{N}}
\newcommand{\rset}{\mathbb{R}}

\newcommand{\floor}[1]{\lfloor{#1}\rfloor}

\newcommand{\PP}[1][]{\ifthenelse{\equal{#1}{}}{\ensuremath{\mathbb{P}}}{\ensuremath{\mathbb{P}\left( #1 \right) }}}
\newcommand{\EE}[1][]{\ifthenelse{\equal{#1}{}}{\ensuremath{\mathbb E}}{\ensuremath{{\mathbb E}\left[ #1 \right]}}}
\newcommand{\Var}[1][]{\ifthenelse{\equal{#1}{}}{\ensuremath{\mathbb{V}\mathrm{ar}}}{\ensuremath{\mathbb{V}\mathrm{ar}\left[ #1 \right] }}} 

\DeclareMathOperator{\Span}{span}

\newcommand\ie{\emph{i.e.}\xspace}

\newcommand\Iff{\emph{i.f.f.}\xspace}

\newcommand{\mb}{\mathbf}
\newcommand{\ud}{\,\mathrm{d}}
\newcommand{\point}{\,\cdot\,}
\newcommand{\given}[1][{}]{\;\middle\vert\;{#1} }

\newcommand{\comm}[1]{\ignorespaces}

\newcommand{\rv}{{RV}\xspace}
\newcommand{\fidi}{\emph{fidi}\xspace}

\newcommand{\Lspace}{L^2[0,1]}
\newcommand{\Hilb}{\mathbb{H}}
\newcommand{\Mzer}{M_0}
\newcommand{\Mto}{\xrightarrow[]{\Mzer}}
\newcommand{\wto}{\xrightarrow[]{w}}
\newcommand{\RV}{\mathrm{RV}}
\newcommand{\sphere}{\mathbb{S}}
\newcommand{\Pangt}{P_{\Theta,t}}
\newcommand{\Panginf}{P_{\Theta,\infty}}
\newcommand{\HSH}{\mathrm{HS}(\Hilb)}
\newcommand{\im}{\mathrm{Im}}

\newcommand{\ES}[2]{V^{#1}_{#2}}
\newcommand{\ESh}[2]{\widehat{V}^{#1}_{#2}}
\newcommand{\ev}[2]{\lambda^{#1}_{#2}}
\newcommand{\Cov}[1]{C_{#1}}
\newcommand{\Covb}[1]{\overline{C}_{#1}}
\newcommand{\Covh}[1]{\widehat{C}_{#1}}
\newcommand{\eg}[2]{\gamma^{#1}_{#2}}
\newcommand{\tnk}{t_{n,k}}
\newcommand{\tnkh}{\widehat{t}_{n,k}}

\newcommand{\Pareto}[1]{\mathrm{Pareto}({#1})}
\newcommand{\calN}[1]{\mathcal{N}}
\RequirePackage[colorlinks, linkcolor=DarkBlue, citecolor=blue,urlcolor=blue]{hyperref}

\newtheorem{theorem}{Theorem}[section]
\newtheorem{proposition}{Proposition}[section]
\newtheorem{corollary}{Corollary}[section]
\newtheorem{lemma}{Lemma}[section]
\newtheorem{example}{Example}[section]

\theoremstyle{definition}

\theoremstyle{remark}
\newtheorem{remark}{Remark}[section]

\numberwithin{equation}{section}

\title{Regular Variation in Hilbert Spaces and\\ Principal Component Analysis for Functional Extremes}
\author{Stephan Clémençon, Nathan Huet, Anne Sabourin}
\date{\today}

\begin{document}

\maketitle
\begin{abstract}

Motivated by the increasing availability of data of functional nature, we develop a general probabilistic and statistical framework for
extremes of regularly varying random elements $X$ in $L^2[0,1]$. 
We place ourselves in a Peaks-Over-Threshold framework
where a functional extreme is defined as an observation $X$ whose $L^2$-norm $\|X\|$ is comparatively large.
Our goal is to propose a dimension reduction framework resulting into
finite dimensional projections for such extreme observations.  Our 
contribution is double. First, we investigate the notion of Regular
Variation for random quantities valued in a general separable Hilbert
space, for which we propose a novel concrete characterization
involving solely stochastic convergence of real-valued random
variables.  Second, we propose a notion of functional Principal Component Analysis (PCA) 
accounting for the principal `directions' of functional
extremes.
We investigate the statistical properties of the empirical covariance operator  
of the angular component of extreme functions,  by upper-bounding the Hilbert-Schmidt norm of
the estimation error  for finite sample
sizes. Numerical experiments with simulated and real data illustrate
this work.

  \end{abstract}
\tableofcontents

  \section{Introduction}

The increasing availability of data of functional nature and various
applications that could now possibly rely on such observations, such
as predictive maintenance of sophisticated systems (\textit{e.g.}
energy networks, aircraft fleet) or environmental risk assessment (\textit{e.g.} air quality monitoring), open new perspective for Extreme Value
Analysis. In
particular,
massive measurements sampled at an ever finer
granularity
offer the possibility of observing extreme behaviors, susceptible to
carry relevant information for various statistical tasks, \emph{e.g.}
anomaly detection or generation of synthetic extreme
examples.

The main purpose of this paper is to develop a general probabilistic
and statistical framework for the analysis of extremes of regularly varying random functions in the space $L^2[0,1]$, the Hilbert space of
square-integrable, real-valued functions over $[0,1]$, 
with immediate possible generalization to other
compact domains, \emph{e.g.} spatial ones.  A major feature of the
proposed framework is the possibility to project the observations
onto a finite-dimensional functional space,
\emph{via} a modification of the standard functional Principal Component Analysis (PCA)  which is
suitable for heavy-tailed observations, for which second (or first)
moments may not exist.

Recent years have seen a growing interest in the field of Extreme
Value Theory (EVT) towards high dimensional problems, and modern
applications involving ever more complex datasets. A particularly
active line of research concerns unsupervised dimension reduction for
which a variety of methods have been proposed over the past few years,
some of them assorted with non asymptotic statistical guarantees
relying on suitable concentration inequalities. Examples of such
strategies include identification of a sparse support for the limiting
distribution of appropriately rescaled extreme observations
(\cite{goixJMVA,simpson2020determining,meyer2021sparse,drees2021principal,cooley2019decompositions,medina2021spectral}),
graphical modeling and causal inference based on the notion of tail
conditional independence
(\cite{hitz2016one,segers2020one,gnecco2021causal}), clustering
(\cite{EmilieEJS,janssen+w:2020,chiapino2019multivariate}), see also
the review paper \cite{engelke2021sparse}.  In these works, the
dimension of the sample space, although potentially high, is finite,
and dimension reduction is a key step, if not the main purpose, of the
analysis.  On the other hand, functional approaches in EVT have a long
history and are still the subject of recent development in spatial
statistics, see \emph{e.g.} the recent review from
\cite{huser2022advances}.
For statistical applications, typically for spatial extremes, strong
parametric assumptions must be made to make up for the infinite-dimensional nature of the problem. Dimension reduction is then limited
to choosing a parametric model of appropriate complexity and it is not
clear how to leverage dimension reduction tools recently developed for
multivariate extremes in this
setting.
The vast majority of existing works in functional extremes consider
the continuous case, following in the footsteps of seminal works on
Max-stable processes (\cite{de1984spectral,de2006extreme}): the random
objects under study are random functions in the space
$\mathcal{C}[0,1]^d$, $d\in\nset^*$, of continuous functions on the
product unit interval, endowed with the supremum norm. Some exceptions
exist, \emph{e.g.} the functional Skorokhod space ${\bb D}[0,1]^d$
equipped with the $J_1$- topology has been considered in several works
(see \cite{davis2008extreme,HL05} and the references therein), and
upper-semicontinuous functions equipped with the Fell topology are
considered in
\cite{resnick1991random,molchanov2016max,sabourin2017marginal,samorodnitsky2019extremal}. Again,
it is not clear how to perform dimension reduction in these functional
spaces.

In the present paper we place ourselves in the Peaks-Over-Threshold
(POT) framework: the focus is on the limit distribution of
rescaled observations, conditioned upon the event that their norm
exceeds a threshold, as this threshold tends to infinity. In the
continuous case, an extreme observation is declared so whenever its
supremum norm is large, \emph{i.e.} above a high quantile. The
limiting process arising in this context is a Generalized Pareto 
process. In the standard POT framework, the definition of an extreme
event depends on the choice of a norm which may be of crucial
importance in applications. As an example, in air quality monitoring
for public health matters, it may be more relevant to characterize
extreme concentration of pollutants through an integrated criterion
over a full $24$-hours period, rather than through the maximum hourly
record. This line of thoughts is the main motivation behind the work
of \cite{dombry2015functional}, which consider alternative definitions of
extreme events by means of an homogeneous cost functional, which gives
rise to $r$-Pareto processes.  However the observations are
still assumed to be continuous stochastic processes and the framework
is not better suited for dimension reduction than those developed in the previously cited
works.
A standard hypothesis  underlying the POT approach is
regular variation (RV), which, roughly, may be seen as an assumption
of approximate radial homogeneity regarding the distribution of the
random object $X$ under study, conditionally on an excess of the norm
$\|X\|$ of this object above a high radial threshold. An excellent account of regular variation of multivariate random vectors is given in the monographs
\cite{Resnick1987,resnick2007heavy}.
In \cite{hult2006regular} regular variation is extended to measures on arbitrary
complete, separable metric spaces and involves $\Mzer$-convergence
of  measures associated to the distribution of rescaled random
objects.  One characterization of regular variation in this context is
\emph{via} weak convergence of the pseudo angle $\Theta= \|X\|^{-1} X$
and regular variation of the (real-valued) norm $\|X\|$.  Namely the
law of $\Theta$ given that $\|X\|>t$ ($t>0$),  ${\cal L}(\Theta\,|\,\|X\|>t)$, 
which we denote by $\Pangt$, 
 must converge weakly as $t\to\infty$,
towards  a limit probability distribution $\Panginf$ on the unit sphere 
(see \emph{e.g.}
\cite{segers2017polar,davis2008extreme}).  In the present work
we place ourselves in the general regular variation context defined through
$\Mzer$-convergence in \cite{hult2006regular}, and we focus our
analysis on random functions valued in the Hilbert space $L^2[0,1]$,
which has received far less attention (at least in EVT) than the
spaces of continuous, semi-continuous or \emph{c\`ad-l\`ag} functions.  One main advantage of the proposed framework, in addition
to allowing for rough function paths, is to pave the way for
dimension reduction of the observations \emph{via} functional PCA
of the \emph{angular} component $\Theta$. In this respect
the dimension reduction strategy that we propose may be seen as an
extension of \cite{drees2021principal}, who worked in the
finite-dimensional setting and derived finite sample guarantees
regarding the eigenspaces of the empirical covariance operator for
$\Theta$. However their techniques of proof cannot be leveraged in the
present context because they crucially rely on the compactness of the
unit sphere in $\rset^d$, while the unit sphere in an infinite-dimensional Hilbert space is not compact.

Several questions arise. First, when dealing with functional
observations, the choice of the norm (thus of a functional space) is
not indifferent, since not all norms are equivalent. In particular,
their is no reason why regular variation  in one functional
space (say, ${\cal C}[0,1]$) would be equivalent to regular variation
in a larger space such as  $L^2[0,1]$.  Also a recurrent issue
in the context of weak convergence of stochastic processes is to
verify tightness conditions in addition to weak convergence of finite-dimensional projections, in order to ensure weak convergence of the
process as a whole. The case of Hilbert valued random variables makes
no exception (see \emph{e.g.}  Chapter 1.8 in \cite{vaart1997weak}). A
natural question to ask is then: 'What concrete conditions regarding the
angular and radial components in a RV/POT framework, which may be verified in practice on
specific generative examples or even on real data,  are sufficient in
order to ensure tightness?'.  Regarding the
PCA of the angular distribution, one may wonder
whether the eigen functions associated with the angular covariance
operator above finite levels $t>0$ indeed converge to the eigen
functions of the covariance operator associated with the limit distribution $\Panginf$ under the RV
conditions alone, and whether the results of \cite{drees2021principal}
regarding concentration of the empirical eigen spaces indeed extend to
the infinite-dimensional Hilbert space setting.

Extreme Value Analysis of  functional PCA with  $L^2$-valued random functions
have already been considered in the literature,
from a quite different perspective however, leaving the above questions
unanswered.
In \cite{Xiong2}, the authors assume regular variation of the scores
of a principal component decomposition, (\emph{i.e.} the random coordinates of
the observations projected onto a $L^2$-orthogonal family) and they
investigate the extremal behavior of their empirical counterparts. In
\cite{Xiong3} and \cite{kokoszka2023principal}, regular variation is
assumed and various convergence results regarding the empirical
covariance operators of the random function $X$ (not the angular
component $\Theta$) are established, under the condition that the
regular variation index belongs to some restricted interval,
respectively $2< \alpha < 4$ and $0<\alpha<2$. In contrast in the
present work the value of the regular variation index is unimportant
as the PCA that we consider is that of the
\emph{angular component} $\Theta$ of the random functions. As $\Theta$
belongs to a bounded subset of $L^2[0,1]$, existence of moments of any
order is automatically granted.  Also in the existing works above
mentioned, regular variation in $L^2[0,1]$, in the sense of Hult and
Lindskog, is taken for granted and no attempt is made to translate the
general, abstract definition from \cite{hult2006regular} into
concrete, finite-dimensional
conditions. In \cite{kim2021extremal}, extremal dependence between the scores of the  functional PCA of $X$ is investigated. They prove on this occasion  (see Proposition 2.1 therein) that regular variation in $L^2[0,1]$   implies multivariate regular variation of finite-dimensional projections of $X$. However, the reciprocal statement is not investigated.

The contribution of the present article is twofold. $(i)$ We provide a comprehensive
description of the notion of regular variation in a separable  Hilbert 
space which fits into the framework of \cite{hult2006regular}.  In
Section \ref{sec:HT}, we formulate specific characterizations involving finite-dimensional projections and moments of the angular variable $\Theta$,
and we discuss the relationships between regular variation in ${\cal C }[0,1]$ and in
$L^2[0,1]$. It turns out that the former implies the latter, whereas the converse is not true. We 
provide several examples and counter-examples illustrating our
statements.  $(ii)$ We make a first step towards
bridging the gap between dimension reduction approaches and functional
extremes by considering the functional PCA of the angular
variable $\Theta$. In Section \ref{sec:KL}, we investigate the convergence of the non
asymptotic covariance operator associated with distribution
 $\Pangt$.  
In the situation where $n\geq 1$ independent realizations of the random function $X$ can be observed, we additionally provide statistical
guarantees regarding empirical estimation of the sub
asymptotic covariance operator associated to a radial threshold
$t_{n,k}$, the $1-k/n$ quantile of the radial variable $\|X\|$, in the
form of concentration inequalities regarding the Hilbert-Schmidt norm
of the estimation error, which leading terms involve the number $k\leq n$ of
extreme order statistics considered to compute the estimator. These bounds, combined with regular
variation of the observed random function $X$ and the results from the preceding section ensure in particular the consistency of the empirical estimation
procedure. In Section \ref{sec:num} we present experimental results involving
real and simulated data illustrating the relevance of the proposed
dimension reduction framework.
Certain technical details are deferred to the Appendix.

For clarity, we start off by recalling some
necessary background regarding probability and weak convergence in
Hilbert spaces, functional PCA and regular variation.

\section{Background and Preliminaries}\label{sec:background}

As a first go, we recall some key facts about regular variation in metric spaces, probability in Hilbert spaces and  Principal Component Analysis of random elements of a Hilbert space.
Here and throughout, the indicator function of any event $\mathcal{E}$ is denoted by $\1\{\mathcal{E}\}$. The Dirac mass at any point $a$ is written as  $\delta_a$ and  the integer part of any real number $u$ by $\lfloor u \rfloor$. By $\Lspace$ is meant the Hilbert space of square integrable, real-valued  functions $f:[0,1]\rightarrow \mathbb{R}$  equipped with its standard  inner product  $\langle f,\; g \rangle=\int_0^1 f(s)g(s)ds$  and the  $L^2$-norm $\| f\|=(\int_0^1 f^2(s)ds)^{1/2}$. 
Our results are valid in any arbitrary separable Hilbert space $\Hilb$, for which  we abusively use  the same notations regarding  the scalar product and the norm as in the special case of $\Lspace$ when it is clear from the context. Also we write $\Hilb_0 = \Hilb\setminus\{0\}$. 
Finally the arrow $\wto$ stands for weak convergence of Borel probability measures: $P_n\wto P$ \Iff we have $\int f \ud P_n \to \int f \ud P$ as $n\to +\infty$ for any bounded, continuous function $f$ defined on the same space as $P_n$ and $P$.

\subsection{Regular Variation in Euclidean and  Metric Spaces}\label{subsec:reg_var}

We recall here the main features of  Regular Variation in  metric spaces, a framework introduced by \cite{hult2006regular} as a generalization of the Euclidean case documented \emph{e.g.} in \cite{Resnick1987,bingham1989regular}. 

Let $(E, d)$ be a complete separable metric space, endowed with a 
 multiplication by nonnegative real numbers $t>0$, such that the
mapping $(t,x)\in \mathbb{R}_+\times E \mapsto tx$ is
continuous.  One must assume the existence of an  \textit{origin} $\mathbf{0} \in E$,  such that
$0x=\mathbf{0}$ for all $x\in E$. In the present work we shall 
take  $E = \Hilb$, a separable, real Hilbert
space, and $\mb{0}$ will simply be the zero of $\Hilb$.  Let
$E_0= E\setminus \{\mb{0}\}$.  For any subset
$A\subset E$, and $t>0$, we write $tA=\{tx:\; x\in
A\}$.
Denote by $\mathcal{C}_0$ the set of bounded and continuous
real-valued functions on $E_0$ which vanish in some
neighborhood of $\mathbf{0}$ and let $\Mzer$ be the class of Borel
measures on $E_0$, which are finite on each Borel subset of
$E_0$ bounded away from
$\mathbf{0}$.
Then the sequence $\mu_n$ \emph{converges to $\mu$  in $\Mzer$} and we write $\mu_n \Mto \mu$, if 
$\int f d\mu_n \rightarrow \int f d\mu$ for any
$f\in \mathcal{C}_0$.

A measurable \emph{function} $f: \rset\to\rset$ is called \emph{regularly varying} with index $\rho$, and we write $f \in \RV_\rho$, whenever for any $x>0$ the ratio $f(tx)/f(t)\to x^\rho$ as $t\to\infty$. 
A \emph{measure} $\nu$ in  $\Mzer$ is \emph{regularly varying} \Iff there exists a nonzero measure $\mu$ in $\Mzer$ and a regularly varying function $b$ such that
\begin{equation}\label{eq:GRV1}
b(n)\nu(n \point )\Mto \mu(\point) \text{ as } n\rightarrow +\infty.
\end{equation}
The limit measure is necessarily homogeneous,  for all $t>0$ and Borel subset $A$ of $E_0$, $\mu(tA )=t^{-\alpha}\mu(A)$ for some $\alpha>0$. Then we say that $\nu$ is regularly varying  with index $-\alpha$ and we write  $\nu\in\RV_{-\alpha}$.
A \emph{random element} $X$ valued in $E$ (that is, a Borel measurable map from some probability space $(\Omega,\mathcal{A}, \PP)$ to $E$) is called regularly varying with index $\alpha>0$  if its probability distribution is  in $\RV_{-\alpha}$.  In this case, one writes $X\in RV_{-\alpha}(E)$.
A convenient characterization of  regular variation of a random element $X$  is obtained through a polar decomposition.  Let $r(x) = d(x, \mb 0)$ for $x\in E$. For simplicity, and because it is true in the  Hilbert space framework that is our main concern,  we focus on the case  where the distance to $\mb 0$ is homogeneous, although this assumption can be relaxed, as in \cite{segers2017polar}. Notice that in $\Hilb$, $r(x) = \|x\|$. Introduce a pseudo-angular variable, $\Theta = \theta(X)$ where for $x\in E_0$, $\theta(x) = r(x)^{-1} x$ and let $R = r(X)$. Denote by $\sphere$ the unit sphere in $E$ relative to $r$, $\sphere=\{x \in E: r(x) = 1\}$, equipped with the trace Borel $\sigma$-field $\cal B(\sphere)$.
The map $T: E_0\to \rset_+^* \times \sphere : x\mapsto (r(x), \theta(x))$  is the polar decomposition. A key quantity throughout this work will be the conditional distribution of the angle given that $R>t$  for which we introduce the notation
\begin{equation}\label{eq:condit-distrib-angle}
  \Pangt( \point  ) = \PP[\Theta \in \point \,|\, R>t].
\end{equation}

Several equivalent characterizations of regular variation of $X$ have been proposed in \cite{segers2017polar} in terms of the pair $(R,\Theta)$ where $R = r(X)$, thus extending classical characterizations in the multivariate setting, see \cite{resnick2007heavy}.
In particular the next statement  shall prove to be useful in the subsequent analysis.
\begin{proposition}[Proposition~3.1 in~\cite{segers2017polar}]\label{prop:polarRVCond}  
A random element $X$ in $E$ is regularly varying with index $\alpha>0$ \Iff  conditions~\ref{cond_rv_radius} and~\ref{cond_weakcv_angle} below are simultaneously  satisfied: 
\begin{enumerate}[label=(\roman*),ref=(\roman*)]
\item\label{cond_rv_radius} The radial variable $R$ is regularly varying in $\rset$ with index $\alpha$; 
\item \label{cond_weakcv_angle} There exists a probability distribution $\Panginf$ on the sphere $(\sphere,\cal B(\sphere))$  such that $\Pangt\wto\Panginf$ as $t\to\infty$. 
\end{enumerate}
\end{proposition}

\subsection{Probability and Weak Convergence in Hilbert Spaces} \label{subsec:probaH}
Most of the background gathered in this section  may be found with detailed proofs, references and discussions in the monograph \cite{hsing2015theoretical}, which  provides a 
self-contained introduction to mathematical foundations of functional data  analysis. Other helpful  resources regarding probability and measure theory in Banach spaces and Bochner integrals  include \cite{vakhania2012probability} or \cite{mikusinski1978bochner}.

\paragraph*{Probability in Hilbert spaces} Consider a separable Hilbert space
$(\Hilb, \langle \cdot, \cdot \rangle)$ and denote by $\| \cdot\|$ the
associated norm.  Let $(e_i)_{i \s 1}$ be any orthonormal basis of
$\Hilb$.  Since a separable Hilbert space is a particular instance of
a Polish  space it follows from basic measure theory in
 (see \textit{e.g.} \cite{vakhania2012probability},
Theorem 1.2) that the Borel $\sigma$-field $\mathcal{B}(\Hilb)$
is generated by the family of mappings $\{ h^{*}: x\mapsto\langle x, h  \rangle , \; h\in\Hilb\}$, or in other words,  by the class of cylinders
$$\mathcal{C}= \{ (h^*)^{-1}(B),  B\in \mathcal{B}(\rset)\}.$$ In addition,
since the countable family $(e_i^*)_{i\geq 1}$ separates points in
$\Hilb$, it also generates the Borel $\sigma$-field,
see Proposition 1.4 and its
corollary in \cite{vakhania2012probability}. In other words, if we denote by $\pi_N$ the projection
from $\Hilb$ to $\rset^N$ onto the first $N\geq 1$ basis vectors,
$\pi_N(x) = (\langle x , e_1\rangle, \dots, \langle x , e_N\rangle)$,
the family of cylinder sets
$$\tilde{ \mathcal{C}} = \left\{\pi_N^{-1}(A_1\times \dots \times A_N),  A_j \in \mathcal{B}(\rset), j\le N, N \in \nset^*   \right\}$$
also generates $\mathcal{B}(\Hilb)$.
We call \emph{$\Hilb$-valued random element} (or
variable) any Borel-measurable mapping $X$ from a probability space
$(\Omega, \mathcal{A}, \bb{P})$ to $\Hilb$.  A mapping
$X:\Omega\to\Hilb$ is Borel-measurable \Iff the real-valued
projections $\langle X,h \rangle$ are Borel-measurable for all
$h \in \Hilb$, and the distribution of $X$ is entirely characterized by the distributions of these univariate projections, see Lemma
1.8.3. in \cite{vaart1997weak} or Theorem 7.1.2 in
\cite{hsing2015theoretical}. Since the family $\tilde C$ of cylinder sets is a $\pi$-system generating $\mathcal{B}(\Hilb)$, it is also true that the distributions of all finite dimensional  projections $(\pi_N(X), N \in\nset)$ onto a given basis also determine the distribution of $X$.
Integrability conditions for  random elements in $\Hilb$  are understood here in the Bochner sense. A random element  $X:(\Omega, \mathcal{A}, \bb{P}) \to \Hilb$ is Bochner integrable \Iff $\E[ \|X\|] < \infty$. Then the expectation of $X$, denoted by $\E [X]$, is the unique element of $\Hilb$ such that $\langle \E[ X] ,h \rangle = \E[ \langle X,h \rangle]$ for all $h \in \Hilb$.
A key property of the classic expectation is linearity, and is also satisfied  by the expectation defined in the Bochner sense. Namely if $T$ is a bounded, linear operator from  $\Hilb_1$ to $\Hilb_2$, two Hilbert spaces, and if $X$ is a Bochner-integrable random element in $\Hilb_1$ then $T(X)$ is also Bochner-integrable in $\Hilb_2$ and $T(\EE[X]) = \EE[T(X)]$, see Theorem~3.1.7 in \cite{hsing2015theoretical}. 
Many other properties of the classic expectation of a real-valued random variables
are preserved, \textit{e.g.} dominated convergence theorem.  In particular, a version of Jensen's inequality can be formulated for $\Hilb$-valued random variables, see \textit{e.g.} pp. 42-43 in \cite{ledoux1991probability}.
\paragraph*{Weak convergence of $\Hilb$-valued random elements}

As our main concern in Section \ref{sec:HT} is to characterize  regular variation in Hilbert spaces in terms of weak convergence of appropriately rescaled variables,
we recall some basic facts  regarding weak convergence in Hilbert spaces. Most of the material recalled next for the sake of completeness can be found in  Chapter 1.8 of \cite{vaart1997weak} and Chapter~7 of  \cite{hsing2015theoretical} in a more detailed way.

By definition a sequence $( X_n )_{n \in \mathbb{N}}$ of  $\Hilb$-valued random variables \textit{weakly converges} (or \textit{converges in distribution}) to a $\Hilb$-valued random variable $X$, and we  write $X_n \overset{w}{\lra} X$
(or equivalently,  $\mu_n \overset{w}{\lra} \mu$ if $\mu_n$ denotes the probability distribution of $X_n$ and $\mu$, that of $X$),
\Iff, for every bounded, continuous function $f:\Hilb\to \mathbb{R}$, 
we have 
$\E[f(X_n)] \rightarrow \E[f(X)]$
 This abstract definition may be difficult to handle for verifying weak convergence in specific examples. However, weak convergence in $\Hilb$ may equivalently be characterized \emph{via}
 weak convergence of one-dimensional  projections and an asymptotic  tightness condition, as described  next. Notice that, because $\Hilb$ is separable and complete, the Prokhorov Theorem applies, \ie uniform tightness and relative compactness of a family of probability measures are equivalent. Recall that a sequence of probability measures is uniformly tight   if for every $\varepsilon >0 $, there exists a compact set $K \subset \Hilb$ such that $\inf_{n\in\nset} \mu_n(K) \geq 1- \varepsilon$. Notice that, because $\Hilb$ is separable and complete, any single random element valued in $\Hilb$ is tight, see Lemma 1.3.2 in \cite{vaart1997weak}.

 \begin{remark}[On measurability and tightness]\label{rem:measurabilityTightness}
Before proceeding any further, in order to clear out  any potential confusion,
we emphasize that measurability of the considered maps $X_n:\Omega\to\Hilb$ is not required in \cite{vaart1997weak}, while it is assumed in the present work, in which we follow  common practice in functional data analysis focusing on Hilbert-valued observations (as  \emph{e.g} in~\cite{hsing2015theoretical}).  Notice also that the notion of tightness employed in 
\cite{vaart1997weak} as a criterion for relative compactness of a family of random variables $(X_n, n\in\nset)$,  is \emph{asymptotic tightness}, that is: for all  $\varepsilon>0$, there exists a compact subset $K$ of $\Hilb$, such that  for every $\delta>0$,   $\liminf_{n\to \infty} \PP[X_n \in K^\delta]> 1-\varepsilon$. Here, $K^\delta$ denotes the $\delta$-enlargement of the compact set $K$, that is, $\{x\in\Hilb: \inf_{y\in K} \|x-y\| <\delta \}$.  This is seemingly at odds with other presentations (\cite{prokhorov1956convergence,hsing2015theoretical}) where the argument is organized around the standard notion of uniform tightness, recalled above.
However in a Polish space such as $\Hilb$, the two notions of tightness (asymptotic or uniform) are equivalent (\cite{vaart1997weak}, Problem 1.3.9), so that the presentations of \cite{vaart1997weak} and \cite{hsing2015theoretical} are in fact closer together than they may appear at first view. 
 \end{remark}

A convenient criterion which is the main ingredient to ensure  tightness (hence relative compactness) of  a family of random $\Hilb$-valued random variables is termed 
\emph{asymptotically finite-dimensionality} in \cite{vaart1997weak} and seems to originate from \cite{prokhorov1956convergence}. A 
sequence of $\Hilb$-valued random variables 
  is  \emph{asymptotically finite-dimensional} if, given a Hilbert basis $(e_i, i \s 1)$ as above,  for all $\varepsilon,\; \delta >0$, there exists a finite subset $I \subset \bb{N}^*$ such that
\begin{equation}\label{eq:asymptoticFidi}
\limsup_n \PP[ \sum_{i \notin I} \langle X_n, e_i \rangle^2 > \delta ] < \varepsilon.
\end{equation}
It should be noticed that the above property is independent from the
specific choice of a Hilbert basis. Asymptotic finite-dimensionality
combined with uniform tightness of all univariate projections of the
kind $\langle X_n, h\rangle, h\in\Hilb$, is sufficient conditions for  uniform
tightness of the family of random variables $(X_n)_{n \in \mb{N}}$ (see
\cite{hsing2015theoretical}, Theorem 7.7.4). Also, since knowledge of
the distributions of all univariate projections characterizes the
distribution of a random Hilbert-valued variable $X$, asymptotic-finite dimensionality
combined with weak convergence of univariate projections (or of finite dimensional ones on a fixed basis) are
sufficient to prove weak convergence of a family of random elements in $\Hilb$, 
as summarized in the next statement.
\begin{theorem}[Characterization of weak convergence in $\Hilb$]\label{vdv184} 
  A net  of $\Hilb$-valued random variables $(X_t)_{t\in \mathbb{R}}$ converges in distribution to a random variable $X$ if and only if it is asymptotically finite-dimensional and either one of the two conditions below holds:
  \begin{enumerate}
  \item   The net $(\langle X_t,h \rangle)_{t \in \rset_+^*}$ converges in distribution to $\langle X,h \rangle$ for any $h \in \Hilb$; 
  \item The net $(\pi_N(X_t))_{\in\rset_+^*}$ converges in distribution to $\pi_N(X)$ for all $N\in\nset^*$. 
  \end{enumerate}

\end{theorem}
\begin{proof}
  The fact that asymptotic finite-dimensionality together with
  Condition 1. in the statement  imply weak convergence results from Theorem 1.8.4 in
  \cite{vaart1997weak} in the case where all mappings are
  measurable. To see that Condition 1. may be replaced with Condition
  2. in order to prove weak convergence, note that asymptotic
  finite-dimensionality implies uniform tightness in the case of a
  Hilbert space (see Remark~\ref{rem:measurabilityTightness} above).
  Hence, weak convergence occurs if the two subsequential limits coincide. It
  is so because the family of cylinder sets $\tilde{\mathcal{C}}$ is a
  measure-determining class.
\end{proof}

\subsection{Principal Component Analysis of $\Hilb$-valued Random Elements}\label{subsec:KL}

We recall the necessary definitions and mathematical background
underlying principal component decomposition of $\Hilb$-valued random
elements.
A self-contained exposition of the topic may be found
in~\cite{hsing2015theoretical}, Chapter~7. In the sequel we use
indifferently the terminology \emph{principal component decomposition}
or \emph{principal component analysis} (functional PCA or PCA in
short). Because of its optimality properties in terms of $L^2$-error when $\Hilb=\Lspace$,
functional PCA is widely used for a great variety of statistical
purposes in functional data analysis. A standard reference on this
topic is the monograph
\cite{FDA}.

On $\Hilb$ a separable real Hilbert space as above, and for
$(f,g)\in \Hilb^2$ the tensor product $f\otimes g$ is the linear operator
on $\Hilb$ defined by $f\otimes g(h) = \langle f, h\rangle g$. Direct
calculations show that $f\otimes g$ is a Hibert Schmidt operator with
Hilbert-Schmidt norm $\|f\otimes g\|_{\HSH} = \|f\|\|g\|$. We recall
that a linear operator $T$ on $\Hilb$ is \emph{Hilbert-Schmidt} if,
given a Hilbert basis $(e_i)_{i \s 1}$, we have
$\sum_{i\in\nset^*} \|T e_i\|^2 <\infty$. The latter quantity is then
the Hilbert-Schmidt norm of $T$, denoted by $\|T\|_{\HSH}$ and does
not depend on the choice of the Hilbert basis. 
  Hilbert-Schmidt operators are compact and the space $\HSH$ of
  Hilbert-Schmidt operators on $\Hilb$, equipped with
  $\langle\point,\point \rangle_{\HSH}$ the scalar product associated
  with the $\HSH$ norm, is itself a separable Hilbert space. 

Let $X$ be a $\Hilb$-valued random element as above and assume that
$\E\|X\|^2 < \infty$. Then also 
$\EE[ \| X\otimes X\|_{\HSH}]  <\infty$
so that the tensor product inside the expectation is Bochner integrable and one may define the  (non-centered) \emph{covariance operator}
\begin{equation}
  \label{eq:CovarianceOper}
  C =  \EE[ X\otimes X ].
\end{equation}
By construction $C$ is self-adjoint and  $C\in\HSH$, thus $C$ is compact. Also  by linearity of Bochner integration, for any $(h, g) \in\Hilb^2$, we have:
$$
C h = \EE[ \langle h, X \rangle X] \;\text{ and } \langle Ch, g \rangle  = \EE[ \langle h,  X  \rangle \langle X, g\rangle ] . 
$$
A  key result in functional PCA is the eigen decomposition of the covariance operator (see Theorem 7.2.6 from~\cite{hsing2015theoretical} regarding the centered covariance operator, which is also valid for the non-centered one): 
\begin{equation}
  \label{eq:eigenDecompos_covariance}
C = \sum_{i=1}^\infty \lambda_i \varphi_i\otimes \varphi_i  ,
\end{equation}
  where $\lambda_1\ge \lambda_2\ge \dots$ are the eigenvalues sorted by decreasing order and the $\varphi_i$'s are orthonormal eigenvectors. The set of non zero eigenvalues $\lambda_i$ is either finite, or a
  sequence of nonnegative numbers converging to zero. The non zero
  eigenvalues have finite multiplicity.  The eigen functions $\varphi_i$ form a Hilbert basis of  $\overline{\im(C)}$.
As it is the case for the centered version of $C$, the decomposition~\eqref{eq:eigenDecompos_covariance} immediately derives from the
spectral theorem for compact, self-adjoint operators and the fact that $C$ is nonnegative definite.

A useful property of the eigen functions $(\varphi_i)_{i\ge 1}$ is that they allow perfect signal reconstruction, since 
  almost-surely, $X$ may be decomposed as
  \begin{equation}
    \label{eq:KLexp}
  X = \sum_{i = 1}^\infty \langle X,\varphi_i \rangle \varphi_i,    
\end{equation}
see Theorem~7.2.7 in~\cite{hsing2015theoretical}. 
  The \emph{scores} $Z_i  = \langle X , \varphi_i\rangle$ satisfy  $\EE[Z_i^2] = \lambda_i$ and  $\EE[Z_iZ_j]=0$, 
so that the expansion \eqref{eq:KLexp} is called \textit{bi-orthogonal}.  
For all $N\geq 1$, the  truncated  expansion $\sum_{i\leq N}\langle X,\; \varphi_i \rangle \varphi_i$ is \textit{optimal} in the sense that it minimizes  
the integrated mean-squared  error
\begin{equation}\label{eq:meanSqError}
  \E  \Big(\big\| X -\sum_{i=1}^N \langle X,\; u_i\rangle  u_i \big \|^2 \Big)
\end{equation}
over all orthonormal collections $(u_1, \ldots, u_N )$ of $\Hilb$. The tail behavior of the (summable) eigenvalue sequence $(\lambda_i)_{i\geq 1}$ 
describes the optimal $N$-term approximation error, insofar as 
\begin{equation*}
\sum_{i>N}\lambda_i=\E \Big( \big\| X-\sum_{i=1}^N \langle X,\; \varphi_i\rangle  \varphi_i\big\|^2 \Big).
\end{equation*}
Notice that in the present paper we consider non centered covariance operator, mainly for the purpose of alleviating notations. We refer the reader to 
\cite{cadima2009relationships} for a comparison of centered and uncentered PCA.

\begin{remark}[Functional PCA and Karhunen-Loève expansion]
  The \emph{functional PCA} framework is closely related to the
  celebrated \emph{Karhunen-Loève expansion} in the case where $\Hilb = \Lspace$, however both terms refer
  to subtly different frameworks,  which deserves an explanation. The
  former framework (which is the one preferred in this work) relies on
  a $\Hilb$-valued random element $X$, with standard results
  concerning convergence of the expansions of $X$ and its covariance
  operator in the Hilbert norm and Hilbert-Schmidt norm,
  respectively, recalled above. Then $X$'s trajectories are in fact equivalence classes
  of square-integrable functions and the specific value $X_s(\omega)$ of a realisation $X(\omega)$
  at $s\in[0,1]$ is only defined almost everywhere.  In contrast, the latter (Karhunen-Loève) framework   relies
  on a second order stochastic process $X= (X_s, s\in[0,1])$, that is, a collection of random variables,  which is continuous
  is quadratic
  mean with respect to the index $s$. 
  Then one must impose additional joint measurability conditions of
  the mapping $(\omega, s)\mapsto X_s(\omega)$ in order to ensure that
  the process $X$ is indeed a $\Hilb$-valued random
  element.  In such a case the mean functions and the
  covariance operators defined both ways coincide.  Also, the
  celebrated Karhunen-Loève Theorem (\cite{loeveprobability}) ensures
  convergence in quadratic mean of the expansion of $X_s$, uniformly
  over $s\in[0,1]$. In order to avoid another layer of
  technicality, and because our main interest indeed lies in the
  eigenspaces of covariance operators rather than in pointwise
  reconstruction of the
  functions, 
  we adopt in the present work the view where $X$ is a $\Hilb$-valued
  random
  element, although additional joint measurability assumptions may be imposed in order to fit into the Karhunen-Loève framework. 
\end{remark}

\section{Regular Variation in Hilbert Spaces}\label{sec:HT}


As a warm up we discuss a classic example in EVT, a multivariate multiplicative model within the framework of the multiplicative Breiman's lemma (\cite{basrak2002regular}, Proposition~A.1) for which \rv may be easily proved using existing general characterizations such as Equation~\eqref{eq:GRV1}. This example will serve as a basis for our simulated data example in Section~\ref{sec:num}.

\begin{example}\label{prop:example1}
  Let  $Z = (Z_1,\ldots, Z_d) \in \rset^d$ be regularly varying  with index $\alpha>0$ and limit measure $\mu$,  and
  let $A = (A_1,\dots,A_d)$ be  a random vector of $\Hilb$-valued variables $A_i$, independent of  $Z$,  such that $\EE[(\sum_{j=1}^d \|A_j\|_{\Hilb}^2) ^{\gamma/2}] < \infty$ for some $\gamma>\alpha$.  Then,
  $$
  X  = \sum_{j=1}^d Z_j A_j
  $$
  is regularly varying  in $\Hilb$ with limit measure
  $\tilde \mu(\point) = \EE[\mu\{x \in \rset^d: \sum_{j=1}^d A_j x \in (\point) \}]$.
\end{example}
\begin{proof}
  In their  Proposition A.1,  \cite{basrak2002regular} consider  the case where $A_j \in \rset^q$ and $\mb A = (A_1,\ldots, A_d)$ is a $q\times d$ matrix. In the proof, they use the operator norm for $A$, but because all norms are equivalent in that case, their argument remains
  valid with the finite-dimensional Hilbert-Schmidt norm.
  In this finite-dimensional context,  $\|A\|$ is equal to $ (\sum_{j=1}^d \|A_j\|_2^2)^{1/2}$, where $\|\point\|_2$ is the Euclidean norm. 
An inspection of the  arguments in their  proof
shows that they also apply to the case where $A_j \in  \Hilb$, up to replacing  $\|A_j\|_2$ with $\|A_j\|_{\Hilb}$ and  $\|A\|$ with $(\sum_{j=1}^d \|A_j\|_{\Hilb}^2)^{1/2}$.  In particular Pratt's lemma is applicable because Fatou's Lemma is valid for nonnegative Hilbert space valued functions. 
\end{proof}

The remainder of this section  aims at providing  some insight on specific properties of   \rv in $\Hilb$, as compared with \rv in general separable metric spaces as introduced by \cite{hult2006regular} or, at the other end of the spectrum, \rv in a  Euclidean space.  On the one hand, we focus on possible finite-dimensional characterizations of \rv in $\Hilb$, 
with a view towards statistical applications in which abstract convergence conditions in an infinite dimensional space cannot be verified on real data, 
while finite-dimensional conditions may serve as a basis for statistical tests. Although  we do not go as far as proposing such rigorous statistical procedures, we do suggest in the experimental section some convergence diagnostics relying on the results gathered in this section.  On the other hand we discuss the relationships existing between \rv in $\mathcal{C}[0,1]$ and \rv in $\Hilb = L^2[0,1]$.

\subsection{Finite-dimensional Characterizations of Regular Variation in $\Hilb$}
\rv random elements in $\Hilb$ have been present in the literature for
a long time, due to strong connections between \rv and domains of
attraction of stable laws in general and in separable Hilbert spaces
in particular. As an example \cite{kuelbs1973domains} show (through
their Lemma 4.1 and their Theorem 4.11) that a random element in
$\Hilb$ which is in the domain of attraction of a stable law with type
$0<\alpha<2$ is regularly varying. However this connection does not
yield any finite-dimensional characterization which are our main focus
here.

As a first go we recall Proposition 2.1 from \cite{kim2021extremal} making a first connection between regular variation in $\Hilb$ and regular variation of finite dimensional (\fidi in abbreviated form) projections. Let $(e_i,i\in\nset)$ be  a complete orthonormal system in $\Hilb$. For ${\cal I} = (i_1,\ldots i_ N)$    a finite set of indices with cardinality $N\geq 1$, denote by $\pi_{\cal I}$ the `coordinate projection' on the associated finite family, $\pi_{\cal I}(x) = (\langle x, e_{i_1} \rangle,\ldots, \langle x, e_{i_N}  \rangle  ), x\in\Hilb$. In particular we denote by $\pi_N:\Hilb\to \rset^N$ the projection onto the $N$ first elements of the basis $(e_i,i\in\nset)$.
\begin{proposition}[\rv in $\Hilb$ implies multivariate \rv of \fidi projections]\label{prop:jointImpliesFIdi}
  If $X$ a random element of $\Hilb$ is regularly varying with index $\alpha>0$ then also for all finite index set $\mathcal I$ of size $N\geq 1$, $\pi_{\cal I} X$ is multivariate \rv in $\rset^N$.   
\end{proposition}

One natural question to ask is whether the reciprocal of Proposition~\ref{prop:jointImpliesFIdi} is true.  We answer  in  the negative in Proposition~\ref{prop:CE} below. 

\begin{proposition}[Multivariate \rv of \fidi projections does not imply \rv in $\Hilb$]\label{prop:CE}
  The reciprocal of Proposition~\ref{prop:CE} is not true. In particular there exists a random element $X$ in $\Hilb$ which is not \rv, while
  \begin{enumerate}
  \item \label{prop:CE:1}  for all $N\in\nset^*$, $\pi_N X$ is multivariate \rv in $\rset^N$ with same index $\alpha>0$ ;
  \item \label{prop:CE:2} the norm of $X$ is \rv in $\rset$ with index $\alpha$. 
  \end{enumerate}
\end{proposition}
\begin{proof}[Sketch of proof]
We construct a random element $X$ in $\Hilb$ in such a way that the probability mass of its angular component $\Theta$, given the radial component $R$,  escapes at infinity as $R$  grows. Here, \emph{at infinity} must be understood as $\Span(e_i,  i\ge M)$ as $M\to\infty$. Namely let  $X := R  \Theta$  with radial component $R =\| X \| \sim Pareto(\alpha)$ on $[1, +\infty [$ (\ie $\forall t \s 1, \PP[R_0 \s t] = t^{-\alpha}$) and define the conditional distribution of $\Theta$ given $R$ as the mixture of Dirac masses:   
\begin{equation*}
\mathcal{L}(\Theta | R ) = \frac{1}{\sum_{l=1}^{\floor{R}} 1/l}\sum_{i=1}^{\floor{R}} \frac{1}{i} \delta_{e_i}. 
\end{equation*}
In other words, for $i\le R$, we have  $\Theta = e_i$ with probability proportional to $1/i$. The remaining of the proof, deferred to the Appendix, consists in verifying that $(i)$ all finite-dimensional projections of $X$ are \rv ; $(ii)$ asymptotic finite-dimensionality (see Equation~\eqref{eq:asymptoticFidi}) of the family of conditional distributions $\Pangt$ does not hold, hence it may not converge to any limit distribution, so that Condition~\ref{cond_weakcv_angle} from Proposition~\ref{prop:polarRVCond} does not hold and $X$ may not be \rv. 
  
\end{proof}

The counter-example above suggests that the missing assumption to obtain \rv in $\Hilb$ is some  relative compactness criterion. This is partly  confirmed in the next example where the angular variables $\Theta_t$ is again a mixture model supported by the $e_i$'s but where the probability mass for the conditional distribution of $\Theta$ given $\|  X \|$ concentrates around finite-dimensional spaces. 
The proof, postponed  to the Appendix, proceeds by verifying both conditions from Proposition~\ref{prop:polarRVCond}. 
\begin{example}\label{prop:ex3}
Let $R \sim Pareto(\alpha)$ on $[1,+ \infty[$ and define $\Theta $ through  its conditional distribution given $R=r$, for $r \s 1$, 
\begin{equation}
\mathcal{L}(\Theta | R = r) = \frac{1}{\sum_{l=1}^{\floor{r}} 1/l^2}\sum_{i=1}^{\floor{r}} \frac{1}{i^2} \delta_{e_i}. 
\end{equation}
In words, $\Theta \in \{e_1,e_2,...\}$ and $\forall r \s 1, \forall j \in \bb{N}^*$ such that $ j\m r$, we have $ \PP[\Theta = e_j | R = r] = \frac{1/j^2}{\sum_{l=1}^{\floor{r}} 1/l^2}$.

Then, the random element  $X = R \Theta$ 
is regularly varying in $\bb{H}$ with index $\alpha$ with limit angular random variable $\Theta_\infty$ given by
\begin{equation}
\PP[\Theta_\infty = e_j] = \frac{6}{(\pi j)^2},
\end{equation}
for $j \in \bb{N}^*$.
\end{example}

The next proposition confirms the intuition built up by the above examples that asymptotic finite dimensionality is a necessary additional assumption to \rv of finite dimensional projections and of the norm.

\begin{proposition} \label{prop:RVfidi} Let $X$ be a $\Hilb$-valued
  random element. The three conditions below are equivalent. 
  \begin{enumerate}
  \item $X$ is regularly varying in $\Hilb$ with index $\alpha>0$,
    limit measure $\mu$ and normalizing sequence $b_n>0$, \ie\
    $\mu_n = b_n\PP[X\in n \point] \Mto \mu(\point)$.
  \item The family of measures $(\mu_n)_{n \ge 1}$ is relatively
    compact in $\Mzer(\Hilb)$-topology, and for all $N \in \nset$,
    $\pi_N X$ is regularly varying in $\rset^N$ with limit measure
    $\mu_N = \mu\circ\pi_N^{-1}$, index $\alpha$ and scaling sequence
    $b_n$.
    \item The family of measures $(\mu_n)_{n \ge 1}$ is relatively
    compact in $\Mzer(\Hilb)$-topology, and for all $h \in \Hilb_0$,
    $\langle x,h \rangle $ is regularly varying in $\rset$ with limit measure
    $\mu_h = \mu\circ (h^*)^{-1}$, index $\alpha$ and scaling sequence
    $b_n$, where $h^* (x) = \langle h, x\rangle$. 
  \end{enumerate}
\end{proposition}
\begin{proof}
  \noindent {\bf 1. $\Rightarrow$ 2. and 3.  }
  
  If $X$ is \rv as in the statement 1., then $(\mu_n)_{n\geq 1}$
  converges in the $\Mzer(\Hilb)$ topology and the family is of course
  relatively compact. Also  fix $N\geq 1$ and
  notice that $\pi_N$ is a continuous mapping from
  $(\Hilb, \|\point\|)$ to $\rset^N$ endowed with the Euclidean
  norm.  The same is true for the bounded linear functional  $h^*$.  The continuous mapping theorem in $\Mzer$ (see
  \cite{hult2006regular}, Theorem~2.5) ensures that
  $\mu_n\circ\Pi_N^{-1} \Mto \mu\circ\pi_N^{-1}$ in $\rset^N$ and that
  $\mu_n\circ (h^*)^{-1} \Mto \mu \circ (h^*)^{-1}$ in $\rset$.

  \noindent {\bf 2. $\Rightarrow$ 1.}
If $\mu_n$ is relatively compact, the sequence $\mu_n$
  converges in $\Mzer(\Hilb)$ \Iff any two subsequential limits
  $\mu^1,\mu^2$ coincide. However it follows form the previous 
  implication that in such a case, the finite dimensional projections
  of $\mu^1$ and $\mu^2$ coincide, namely
  $\mu^1\circ\pi_N^{-1} = \mu^2\circ\pi_N^{-1} = \mu\circ\pi_N^{-1}$,
  for all integer $N$.  Consider the family of cylinder sets of
  $\Hilb$ with measurable base,
  $\mathcal C = \{ \pi_N^{-1} (A), A\in \mathcal B(\rset^N), N \in
  \nset^*\}$. On $\mathcal{C}$ the measures $\mu,\mu^1,\mu^2$
  coincide. The cylinder sets family  $\mathcal{C}$ is a $\pi$-system which generates the Borel $\sigma$-field, because it is associated to the family of bounded linear functional $(e_i^*, i\in\nset)$ which separates points. Thus $\mu,\mu^1,\mu^2$ coincide on every Borelian set, and the proof is complete.

  \noindent {\bf 3. $\Rightarrow$ 1.} As above, it is enough to show that two subsequential limits $\mu^1,\mu^2$  coincide. In this case it is obviously so, because it is known that the Borel $\sigma$-field on $\Hilb$ is generated by the mappings $h^*$ (\emph{e.g.} \cite{hsing2015theoretical}, Theorem 7.1.1.).      
\end{proof}

  The line of thought of Proposition~\ref{prop:RVfidi} may  be pursued further by characterizing the property of relative compactness of a  family  $(\nu_n)_{ n\in\nset} \in \Mzer(\Hilb)$ through asymptotic finite-dimensionality (see Equation~\eqref{eq:asymptoticFidi}), following the lines of the proof of  Theorem 4.3 in \cite{hult2006regular}, relying in particular on Theorem 2.6 of the cited reference. 
  However it is also possible to rely  on  known characterizations of relative compactness for probability measures, coupled with the polar characterization of \rv (Proposition~\ref{prop:polarRVCond}). We propose in this spirit the following simple characterization solely based on weak convergence of univariate and finite-dimensional projections, together with regular variation of the norm, without additional requirements regarding asymptotic finite-dimensionality.  

\begin{theorem} \label{thm:thmRVH}
  Let $X$ be a random element in $\Hilb$ and let $\Theta_t$  be a random element in $\Hilb$ distributed on the sphere $\sphere$  according to the conditional angular distribution $\Pangt$. 
  Let $\Panginf$ denote a probability measure on $(\sphere, \mathcal{B}(\sphere))$ and let 
  $\Theta_\infty$ be a random element distributed according to $\Panginf$. 
  The following statements are equivalent. 
\begin{enumerate}
\item \label{thmRVDefault} $X$ is regularly varying with index $\alpha$ with limit angular measure $\Panginf$, so that $\Pangt\wto\Panginf$.
    
\item \label{thmRVH2} $\|X\|$ is regularly varying in $\rset$ with index $\alpha$, and
  $$\forall h \in \bb{H}, \langle \Theta_t, h \rangle \wto \langle \Theta_\infty,h \rangle \quad \text{ as } t\to\infty. $$
    
\item \label{thmRVH-fidi} $\|X\|$ is regularly varying in $\rset$ with index $\alpha$, and
  $$\forall N \in\nset, \pi_N(\Theta_t) \wto \pi_N(\Theta_\infty) \quad \text{ as } t\to\infty. $$
              
\end{enumerate}
\end{theorem}
\begin{proof}
  The fact that \ref{thmRVDefault} implies \ref{thmRVH2} and~\ref{thmRVH-fidi} is a direct consequence of the polar  characterization of \rv (Proposition~\ref{prop:polarRVCond}) and of the continuous mapping theorem applied to the bounded linear mappings  $h^*$, $h\in\Hilb$ and $\pi_N$, $N \in\nset$.
  
 For the reverse implications (\ref{thmRVH-fidi}$\Rightarrow$ \ref{thmRVDefault}) and (\ref{thmRVH2} $\Rightarrow$ \ref{thmRVDefault}), in view of Proposition~\ref{prop:polarRVCond},   
 we only  need to verify that for any sequence $t_n>0$ such that $t_n\to\infty$, $\Theta_{t_n}\wto\Theta_\infty$ in $\Hilb$. From Theorem~\ref{vdv184}, if either Condition~\ref{thmRVH2} or Condition~\ref{thmRVH-fidi} holds true, then it will be so  \Iff the family $P_{\Theta, t_n}, n\in \nset$ is asymptotically finite-dimensional. 
  
  We use the fact, stated and proved in \cite{tsukuda2017change}, that
  if $(Z_n,n\in\nset)$ and  $Z$ are 
  $\Hilb$-valued random elements such that, as
  $n\to\infty$, 
\begin{equation}\label{eq:cvMomentsNorm}
\E[ \|Z_n \|^2] \rightarrow \E[ \|Z\|^2],
\end{equation}
and for all $j\in \mathbb{N}^*$
\begin{equation}\label{eq:cvMomentsProj}
\E[ \langle Z_n,e_j \rangle^2] \rightarrow \E[ \langle Z,e_j \rangle^2],
\end{equation}
then the sequence $(Z_n)_{n\in\nset}$ is asymptotically finite-dimensional.

With $Z_n = \Theta_{t_n}$ and $Z = \Theta_\infty$, Condition~\eqref{eq:cvMomentsNorm} above is immediately satisfied since $\|\Theta_{t_n} \|=  \|\Theta_\infty\| = 1$ almost surely. For the same reason $\EE[  \langle \Theta_{t_n},e_j \rangle^2] = \EE[\varphi(\langle \Theta_{t_n},e_j \rangle)]$,  where $\varphi$ is the bounded, continuous function $\varphi(z) = \min(z^2, 1)$.  Thus, weak convergence of the projections 
$\langle \Theta_{t_n},e_j \rangle$ (Condition~\ref{thmRVH2} or~\ref{thmRVH-fidi} from the statement) together with the continuous mapping theorem 
 imply~\eqref{eq:cvMomentsProj}, which concludes the proof. 
\end{proof}

\subsection{Regular Variation in $\Lspace$ \textit{vs} Regular Variation in  $\mathcal{C}[0,1]$}\label{sec:rvL2-C}
Turning  to the case where $\Hilb = \Lspace$, we discuss the
relationships between the notions of regular variation in $\Lspace$
and in $\mathcal{C}[0,1]$, the space of continuous functions on $[0,1]$. Indeed, any continuous stochastic process
$(X_t, t \in [0,1])$ is also a random element in $\Hilb = \Lspace$, as
proved in~\cite{hsing2015theoretical}, Theorem~7.4.1, or  7.4.2. 
It is thus
legitimate to ask whether regular variation with respect to one norm
implies regular variation for the other norm for such stochastic
processes. 

\begin{proposition}\label{prop:RVC0=>RVL2}
Let $X$ be a continuous process over $[0,1]$. Assume that $X \in RV_{-\alpha}(\mathcal{C}[0,1])$, with  $\mathcal{L}(X/\|X\|_\infty | \|X\|_\infty > t)\rightarrow \mathcal{L}(\Theta_{\infty,\infty})$, as $t \rightarrow +\infty$, where $\Theta_{\infty,\infty}$ is the angular limit process w.r.t. the sup-norm $\|\cdot \|_\infty$.
Then, $X \in RV_{-\alpha}( \Lspace )$, and the  angular limit process $\Theta_{\infty,2}$ (w.r.t. the $L^2$ norm  $\|\cdot\|$) has distribution given by 
\begin{equation}\label{eq:RVC0=>RVL2}
\PP[\Theta_{\infty,2} \in B ] = \frac{\E \big[ \| \Theta_{\infty,\infty} \|^\alpha \1\{\Theta_{\infty,\infty} / \|\Theta_{\infty,\infty}\| \in B \} \big]}{\E \big[ \| \Theta_{\infty,\infty} \| ^\alpha \big]},
\end{equation}
where $B \in \mathcal{B}(\sphere_{2})$. 
\end{proposition}
\begin{proof}
  Since $\|\cdot\|$ is continuous w.r.t. $\|\cdot\|_\infty$ in
  $\mathcal{C}[0,1]$, Theorem~3 in \cite{dombry2015functional}
  applies (upon chosing $\ell(X) = \|X\|$ with the notations of the
  cited reference), which yields regular variation of $X$ in $\Lspace$,
  together with the expression given in~\eqref{eq:RVC0=>RVL2} for the
  angular measure associated with the $L^2$ norm
  $\|\point\|$. 
\end{proof}
One may wonder whether  the reciprocal is also true, \ie if $X \in \mathcal{C}[0,1]$, and  $X \in RV_{-\alpha}(\Lspace)$, is it necessarily the case  that $X \in RV_{-\alpha}(\mathcal{C}[0,1])$? 
A counter-example is given in the next proposition.

\begin{proposition}\label{prop:CE2} The reverse statement of Proposition~\ref{prop:RVC0=>RVL2} is not true in general. There exists a sample-continuous stochastic process over $[0,1]$ which is regularly varying in $\Lspace$ but not in $\mathcal{C}[0,1]$.
\end{proposition}
\begin{proof}
  We construct a `spiked' continuous process with controlled $L^2$
  norm, while the sup-norm is super-heavy
  tailed.  
  Let $Z$ follow a Pareto distribution with parameter $\alpha_Z>0$, and define a   sample-continuous stochastic process
  $$Y(t) = \Big( 1 - \frac{t}{3Z^2\exp(-2Z)}\Big)\exp(Z) \1\{[0,3Z^2\exp(-2Z)[\}.$$
  Straightforward computations yield $\|Y\|_\infty = \exp(Z)$ and
  $\|Y\|_2 = Z$.  Let $\rho$ be another independent Pareto-distributed
  variable with index $0<\alpha_\rho<\alpha_Z$.  Finally, define
  $X = \rho Y$. Then $X$ is a sample-continuous stochastic process
  over $[0,1]$. We have $\|X\|_\infty = \rho \exp(Z)$, which is
  clearly not regularly varying because  (see \emph{e.g.}~\cite{mikosch1999regular}, Proposition 1.3.2)
  $\E [\|X\|_\infty^\delta] = +\infty$ for all $\delta >0$. Thus, $X$
  is not regularly varying in $(\mathcal{C}[0,1],\|\cdot \|_\infty)$.

  On the other hand, the pair $(\rho,Y)$ satisfies the assumptions of
  Example~\ref{prop:example1} with $d=1$. Hence, $X = \rho Y$ is
  regularly varying in $\Hilb = \Lspace$.

\end{proof}

Propositions~\ref{prop:CE2} and~\ref{prop:RVC0=>RVL2} together show  that the framework of $L^2$- regular variation encompasses a wider classes of continuous processes than standard $\mathcal{C}[0,1]$ regular variation. This opens a road towards applications of EVT in situations where the relevant definition of an extreme event has to be understood in terms of `energy' of the (continuous) trajectory, as measured by the $L^2$ norm, rather than in terms of sup-norm.

\section{Principal Component Analysis  of Extreme Functions}\label{sec:KL}
This section gathers the main results of the paper.  Motivated by
dimension reduction purposes, our goal is to construct a finite-dimensional representation of extreme functions. In other words our
primary purpose is to learn a finite-dimensional subspace $V$ of
$\Hilb=\Lspace$ such that the orthogonal projections of extreme
functions onto $V$ are optimal in terms of angular reconstruction
error. Throughout this section we place ourselves in the setting of regular variation  introduced in Section~\ref{sec:HT} and consider a regularly varying random element $X$ in $\Hilb$ as in  Theorem~\ref{thm:thmRVH}, with the same notations.   Our focus is thus on building a low-dimensional representation
of the angular distribution of extremes $\Panginf$ introduced in
Section~\ref{subsec:reg_var}.  We consider the eigen decomposition of
the associated covariance operator
$$
\Cov{\infty} = \EE[\Theta_\infty\otimes \Theta_\infty] = \sum_{j \in \nset} \ev{j }{\infty} \varphi_j^\infty \otimes \varphi_j^\infty, 
$$
where $\Theta_\infty\sim \Panginf$, and the $\varphi_j^\infty$'s and $\ev{j}{\infty}$'s are  eigenfunctions and eigenvalues of $\Cov{\infty}$ following the notations of Section~\ref{subsec:KL}. If $\Panginf$ is sufficiently concentrated around a finite-dimensional subspace of moderate dimension $p$, a  reasonable approximation of $\Panginf$ is provided by its image measure \emph{via} the projection onto $\ES{p}{\infty} = \mathrm{Vect}(\varphi_j^\infty, j\le p)$. Independently from such sparsity assumptions,  the space $\ES{p}{\infty}$ minimizes the reconstruction error \eqref{eq:meanSqError} of the orthogonal projection relative to $\Theta_\infty$. It is also the unique minimizer as soon as  $\ev{p}{\infty}>\ev{p+1}{\infty}$, as discussed in the background section \ref{subsec:KL}.   

Our main results bring finite sample guarantees regarding an empirical
version of $\ES{p}{\infty}$ constructed using the $k\ll n$ largest
observations.  In this respect our work may be seen as an extension of
\cite{drees2021principal}, who consider finite dimensional
observations $X\in\rset^d$, to an infinite dimensional ambient
space. However our proof techniques are fundamentally different from
the cited reference. Indeed their analysis relies on Empirical Risk
Minimzation  arguments  
relative to the reconstruction risk at infinity,
$R_\infty(V) = \lim_{t\to\infty} \EE[ \Theta - \Pi_V \Theta \given R>t]$
where $\Pi_V$ denotes the orthogonal projection onto $V$. The main
ingredients of their analysis are  $(i)$ the fact that $\ES{p}{\infty} $ minimizes the risk at infinity 
$(ii)$ compactness of the unit sphere (or of any bounded, closed subset of
$\rset^d$). In
the present setting such compactness properties do not hold and we follow an entirely different path, as we investigate the convergence of an empirical version of $\Cov{\infty}$ in the Hilbert-Schmidt norm, and then rely on perturbation theory for covariance operators in order to control the deviations of its eigenspaces. We thus consider the pre-asymptotic covariance operator 
\begin{equation}
  \label{eq:preasymptot-covar}
\Cov{t} = \EE[\Theta  \otimes \Theta \given R>t ] = \EE[\Theta_t\otimes \Theta_t].   
\end{equation}

In the sequel, the discrepancy between finite dimensional linear subspaces of $\Hilb$ is measured in terms of the Hilbert-Schmidt norm of the difference between orthogonal projections, namely  we define a distance $\rho$ between finite dimensional subspaces $V,W$  of $\Hilb$,  by
\begin{equation*}
  \rho (V,W) = \| \Pi_V - \Pi_W \|_{HS(\Hilb)}.   
\end{equation*} 
It should be noticed that  \cite{drees2021principal} denote by
$\rho$ the operator norm of the difference between the projections,  which is coarser than the Hilbert-Schmidt one.

We show in  Section~\ref{sec:proba} that 
the first $p$  eigenfunctions of the pre-asymptotic operator $\Cov{t}$ generate a vector space $\ES{p}{t}$ converging to 
$\ES{p}{\infty}$ whenever $\ev{p}{\infty} > \ev{p+1}{\infty}$. Second, we establish in Section~\ref{sec:stats} the consistency of the  empirical subspace $\ESh{p}{t}$ (the one generated by the first $p$ eigenfunctions of  an empirical version of $\Cov{t}$) and we derive nonasymptotic guarantees for its deviations, based on concentration inequalities regarding  the empirical covariance operator.

\subsection{The Pre-asymptotic Covariance Operator and its Eigenspaces} \label{sec:proba}

Since perturbation theory  allows to control the deviations of eigenvectors and eigenvalues of a perturbed covariance operator, a natural first step in our analysis is to  ensure that the pre-asymptotic operator $C_t$ introduced in~\eqref{eq:preasymptot-covar} may indeed be seen as a perturbed version of the asymptotic operator $\Cov{\infty}$, as shown next.

\begin{theorem}[Convergence of the pre-asymptotic covariance operator] \label{thm:propproba}
In the setting of Theorem~\ref{thm:thmRVH}, as $t\to\infty$, the following convergence in the Hilbert-Schmidt norm holds true, 
\begin{equation*}
\|\Cov{t} - \Cov{\infty} \|_{HS(\Hilb)} \rightarrow 0\, .
\end{equation*}
\end{theorem}
\begin{proof}
  Recall from Proposition~\ref{prop:polarRVCond} that \rv of $X$ implies weak convergence of  the net $\Theta_t$ towards $\Theta_\infty$. Using the fact that  the mapping $h \in \Hilb \mapsto h \otimes h \in HS(\Hilb)$ is continuous, also 
  $\Theta_t \otimes \Theta_t$ converges weakly towards $\Theta_\infty \otimes \Theta_\infty$.
Let $(t_n)_{n\in\nset}$ be a nondecreasing sequence of reals converging to infinity.

Since the separability of $(\Hilb,\langle \cdot, \cdot \rangle)$ implies the separability of $(HS(\Hilb),\langle \cdot, \cdot \rangle_{HS(\Hilb)})$ (see \cite{blanchard2007statistical}, Section~2.1), we may apply  the Skorokhod's Representation theorem to the weakly converging sequence $\Theta_{t_n} \otimes \Theta_{t_n}$. Thus there is a  
probability space $(\Omega', \mathcal{F}, \PP')$, and random elements $Y_n$, $n\ge 1$ and  $Y_\infty$ in $HS(\Hilb)$ 
defined on the probability space $(\Omega', \mathcal{F},
\PP')$,  such that $\Theta_{t_n} \otimes \Theta_{t_n} \overset{d}{=} Y_n$, $\Theta_\infty \otimes \Theta_\infty \overset{d}{=} Y_\infty$ and $Y_n$ converges to $Y_\infty$ almost surely with respect to $\PP'$.  

A Jensen's type inequality in Hilbert spaces (see \cite{ledoux1991probability}, pp. 42-43) yields $\|\Cov{t_n} - \Cov{\infty} \|_{HS(\Hilb)} \m \E [\|Y_{n} - Y_\infty \|_{HS(\Hilb)}]$.  The dominated convergence theorem applied  to the vanishing sequence of random variables $\|Y_n-Y_\infty\|_{HS(\Hilb)}$ (which are bounded by the constant $2$) completes the proof.

\end{proof}
\begin{remark}
  An alternative way to obtain  the weak convergence of $\Theta_t\otimes\Theta_t$, which is key in the proof of Theorem~\ref{thm:propproba}, is to leverage Proposition~3.2 in~\cite{Xiong3}, which ensures that the operator     $X \otimes X$ is regularly varying in $HS(\Hilb)$. Since  $\Theta\otimes \Theta$ is indeed the angular component of $X\otimes X$, the result follows by an application of Proposition~\ref{prop:polarRVCond}.
\end{remark}

The next result concerns the  convergence of eigenspaces and is obtained by combining tools  from operator perturbation theory  with the result from Theorem~\ref{thm:propproba}.
In order to avoid additional technicalities  we consider in the next statement an  integer $p$  such that 
$\ev{p}{\infty} > \ev{p+1}{\infty} \ge 0$, that is, a positive  the spectral gap. 
Notice that such a $p$ necessarily exists since 
$\|\Cov{\infty}\|_{HS(\Hilb)}^2 = \sum_{j=1}^\infty (\ev{j}{\infty})^2 < \infty$.

\begin{corollary}[Convergence of pre-asymptotic eigen spaces] \label{cor:CCLproba}
Let  $p\in\nset^*$ be  such that $\ev{p}{\infty}>\ev{p+1}{\infty}$. Then,  as  $t$ tends to infinity, 
\begin{equation*}
\rho(\ES{p}{t}, \ES{p}{\infty}) \rightarrow 0. 
\end{equation*}

\end{corollary}
\begin{proof}
According to Theorem~3 in \cite{zwald2005convergence},  for $A$ and $B$ two Hilbert-Schmidt operators on a separable Hilbert space, and an integer $p$ such that the ordered eigenvalues of  $A$ satisfy  $\ev{p}{}(A)>\ev{p+1}{}(A)$,  if   $\|B\|_{HS(\Hilb)} < \eg{p}{} := \frac{\ev{p}{}(A)-\ev{p+1}{}(A)}{2}$ is such that $A+B$ is still a positive operator, then  following inequality holds
\begin{equation*}
\rho(V^p,W^p) \m \frac{\|B\|_{HS(\Hilb)}}{\eg{p}{}},
\end{equation*}
where $V^p$ and $W^p$ are respectively the eigen spaces spanned by the first $p$ eigenvectors of $A$ and $A+B$. From  Theorem~\ref{thm:propproba}, the operators  $A =  \Cov{\infty}$ and $B = \Cov{\infty} - \Cov{t}$ satisfy the required assumptions stated above for $t$ sufficiently large,  and $\|B\|_{HS(\Hilb)}$ may be chosen arbitrarily small, which  concludes the proof.
\end{proof}

\begin{remark}[Convergence of eigenvalues and choice of $p$]\label{rem:choice_p_cvEigenVals}
  Even though the eigenvalues of  $\Cov{\infty}$ are not the main focus of our work, they are involved  in the conditions  of Corollary~\ref{cor:CCLproba} through the requirement of a  positive spectral gap. Of course these eigen values are unknown, however Weyl's inequality (see \cite{hsing2015theoretical}, Theorem~4.2.8) ensures that $\sup_{j\ge 1} |\ev{j}{t}  - \ev{j}{\infty}| \le \| \Cov{t} - \Cov{\infty} \|_{HS(\Hilb)}$. Identification of an integer $p$ for which the eigen gap is positive may thus be achieved using consistent estimates of the $\ev{j}{t}$'s  for $t$ large enough.

\end{remark}

\subsection{Empirical Estimation: Consistency and Concentration Results}\label{sec:stats}
We now turn to statistical properties of empirical estimates of $C_t$ and its eigen decomposition based on an independent sample   $X_1,...,X_n$ distributed as $X$. 
Following standard practice in Peaks-Over-Threshold analysis, we consider a fixed number of excesses $k$ above a random radial threshold chosen as the empirical $1-k/n$ quantile of the norm, with $k\ll n$. Even though our main results are of non asymptotic nature, letting $k,n\to \infty$ with $k/n\to 0$ yields consistency guarantees such as Corollary~\ref{cor:CCLstats} below.  Denote by $X_{(1)}, \ldots X_{(n)}$ the permutation of the sample such that $\|X_{(1)}\| \s \|X_{(2)}\|\s ... \s \| X_{(n)}\|$, and accordingly, let  $\Theta_{(i)}, R_{(i)}$ denote the angular and radial components of $X_{(i)}$. Then  $\|X_{(k)}\|  = R_{(k)}$ is an empirical version of the $(1-k/n)$ quantile of the norm $R$, which we shall sometimes denote by $\tnkh$.

With these notations an empirical version of $\Cov{t_{n,k}}$ is

\begin{equation}
\Covh{k} := \frac{1}{k }\sum_{i=1}^k \Theta_{(i)} \otimes \Theta_{(i)}. 
\end{equation}

\begin{remark}{\sc (Choice of $k$)} Choosing the number $k$  of observations  considered as extreme,   is an important but difficult topic in EVT. A wide variety of methods have been proposed in univariate problems (\cite{caeiro2016threshold,scarrott2012review}),  some rule of thumbs exist in multivariate settings based on visual inspection of angular histograms (\cite{coles1994statistical} or stability under rescaling of the radial distribution (\cite{stuaricua1999multivariate}) with little theoretical foundations. 
We leave this question outside our scope, although visual diagnostics are proposed in our numerical study based on Hill plots and convergence checking based on the finite-dimensional characterizations of \rv stated in Theorem~\ref{thm:thmRVH}.
\end{remark}

Our analysis of the statistical error $\|\Covh{k} - \Cov{\tnk}\|_{HS(\Hilb)}$  involves the intermediate pseudo empirical covariance
\begin{equation*}
   \Covb{t}:= \frac{1}{\PP[\|X_1\|\s t]} \frac{1}{n}\sum_{i=1}^n \Theta_i\otimes \Theta_i \1 \{R_i \ge t \}.
\end{equation*}
evaluated at $t=\tnk$. Since $\tnk$ is unknown, $\Covb{\tnk} = k^{-1} \sum_{i=1}^n \Theta_i\otimes \Theta_i \1 \{R_i \ge \tnk \}$ is not observable, although its deviation from $\Covh{k}$ may be controlled by the classical Bernstein inequality (Proposition~\ref{prop:prop4.2}). 
Our point of departure is the following decomposition of the statistical error, 
\begin{equation}\label{eq:errorDecompos}
\| \Covh{k}-\Cov{\tnk}\|_{HS(\Hilb)} \le \| \Covb{\tnk}-\Cov{\tnk}\|_{HS(\Hilb)} + \| \Covh{k}-\Covb{\tnk}\|_{HS(\Hilb)}.
\end{equation}
We analyze separately the two terms in the right-hand side of \eqref{eq:errorDecompos} in the next two propositions.

\begin{proposition} \label{prop:prop4.1}
Let $\delta \in (0,1)$. With probability larger than $1- \delta / 2$, we have
\begin{align*}
  \|\Covb{\tnk} - \Cov{\tnk}\|_{HS(\Hilb)}
  &\le \frac{ 1 + 4\sqrt{\log(2/\delta)}}{\sqrt{k}} + \frac{8 \log(2/\delta)}{3 k}
\end{align*}
\end{proposition}
\begin{proof}[sketch of proof]
A Bernstein-type concentration inequality from \cite{mcdiarmid1998concentration} which is applicable to arbitrary functions of $n$ variables with controlled conditional variance and conditional range (Theorem 3.8 of the reference,  recalled in Lemma~\ref{lem:Bernstein} from the  Appendix) ensures that $$\PP \big(\| \Covb{\tnk} - \Cov{\tnk}\|_{HS(\Hilb)} - \mathbb{E}[ \; \|\Covb{\tnk} - \Cov{\tnk}\|_{HS(\Hilb)}\;]   \ge \varepsilon \big) \le \exp\Big(\frac{ -k\varepsilon^2}{4(1+ \varepsilon/3)} \Big).$$ In order to control the expected deviation 
$\EE[\| \Covb{\tnk}-\Cov{\tnk}\|_{HS(\Hilb)}]$ in the left-hand side,   we use the fact that, if $A_1,...,A_n$ are independent centered $\Hilb$-valued random elements,  $\E[\big\|\sum_{i=1}^n A_i \big\|^2]=\sum_{i=1}^n \EE[\| A_i \|^2]$ (Lemma~\ref{lem:pythagoreInExpectation} in the Appendix). We apply this result to $A_i$ chosen as the deviation of the operator  $ 
\Theta_i\otimes \Theta_i \1\{R_i\ge \tnk\}$ from its expectation, which yields 
$$ \E [\| \Covb{\tnk}-\Cov{\tnk}\|_{HS(\Hilb)}]  \le 1 \slash \sqrt{k}$$ (Lemma~\ref{lem:BoundExpectedDeviation}) and finishes the proof, as detailed in Appendix~\ref{sec:proofsPCA}. 
\end{proof}
We now turn to the second term $\| \Covh{h}-\Covb{\tnk}\|_{HS(\Hilb)}$ in the error decomposition~\eqref{eq:errorDecompos}. 
\begin{proposition}\label{prop:prop4.2}
Let $\delta \in (0,1)$. With probability larger than $1-\delta/2$, we have
\begin{align*}
  \|\Covh{k} - \Covb{\tnk} \|_{HS(\Hilb)} & \le \sqrt{\frac{8\log(4/\delta)}{k}}  + \frac{ 4\log(4/\delta)}{3k}\;.\\
\end{align*}
\end{proposition}
\begin{proof}
First, the triangle inequality yields
\begin{align*}
  \| \Covh{k} - \Covb{\tnk} \|_{HS(\Hilb)}
  &= \frac{1}{k}\; \Big\| \; \sum_{i=1}^n \Theta_i \otimes \Theta_i ( \1\{ R_i \ge \tnk \} -  \1\{ R_i \ge \tnkh \} ) \; \Big\|_{HS(\Hilb)}  \\
  &\le \frac{1}{k} \;\sum_{i=1}^n |\; \1\{ R_i \ge \tnk \} -  \1\{ R_i \ge R_{(k)} \}\;| . \\
\end{align*}
The number of non-zero terms inside the sum in the above display is the number of indices $i$ such that ` $R_i < R_{(k)}$ and $R_i\ge \tnk$' , or the other way around, thus 
\begin{align*}
   \| \Covh{k} - \Covb{\tnk} \|_{HS(\Hilb)} &\le \frac{1}{k} \; \Big| \;\sum_{i=1}^n  \1\{ R_i \ge \tnk \} -  k \;\Big|. 
\end{align*}
Notice that $\sum_{i=1}^n  \1\{ R_i \ge \tnk \}$ follows a  Binomial distribution  with  parameters $(n, k/n)$.
The (classical) Bernstein's inequality as stated \emph{e.g.} in \cite{mcdiarmid1998concentration}, Theorem~2.7,  yields
\begin{equation*}
  \PP[\| \Covh{k} - \Covb{\tnk} \|_{HS(\Hilb)} \s \varepsilon ]
  \le \PP\Big(\; \Big| \sum_{i=1}^n  \1\{ R_i\ge \tnk \} -  k \Big| \s k\varepsilon \; \Big)
  \le 2 \exp \Big( \frac{-k\varepsilon^2}{2(1+\varepsilon/3)} \Big). 
\end{equation*}
Solving for $\varepsilon$ and using the fact that $\sqrt{a+b} \le \sqrt{a}+\sqrt{b}$ for any nonnegative numbers $a,b$, we obtain  the upper bound in the statement. 
\end{proof}
We are now ready to state a non-asymptotic guarantee regarding the deviations (in the HS-norm) of the  empirical covariance operator. 
\begin{theorem} \label{thm:stats}
Let $\delta \in (0,1)$. With probability larger than $1-\delta$, we have
\begin{align*}
  \|\Covh{k} - \Cov{\tnk} \|_{HS(\Hilb)}
  & \le \frac{ 1 + 4\sqrt{\log(2/\delta)} + \sqrt{8\log(4/\delta) }}{\sqrt{k}} +
    \frac{8 \log(2/\delta) + 4\log(4/\delta)}{3 k}
\end{align*}
\end{theorem}
\begin{proof}
  Observe that the following inclusion  between adverse events holds true because of \eqref{eq:errorDecompos}, 
  \begin{align*}
   \big\{ \| \Covh{k}-\Cov{\tnk} \|_{HS(\Hilb)} \s  \varepsilon_1+\varepsilon_2 \big\}\; 
    & \subset \big\{ \|\Covh{k} - \Covb{\tnk} \|_{HS(\Hilb)} \ge \varepsilon_1 ]
      \big \}\; \cup \; 
      \big\{ \| \Covb{\tnk}-\Cov{\tnk} \|_{HS(\Hilb)} \s \varepsilon_2  \big \}, 
  \end{align*}
  for all $\varepsilon>0$. 
  A union bound and Propositions~\ref{prop:prop4.1},~\ref{prop:prop4.2} conclude the proof. 
\end{proof}

\begin{remark}[Tightness of the upper bound, asymptotics]\label{rem:tightnessBound}
  The bound obtained in Theorem~\ref{thm:stats} constitutes a minimal guarantee regarding covariance estimation of the extremes. By no means do we claim optimality regarding the multiplicative constants, which we have not tried to optimize, as revealed by an inspection fo the proof where the decomposition of the adverse event into two events of same probability may be sub-optimal.  However the leading term of the error as $k\to\infty$ is an explicit, moderate  constant 
  and the rate of convergence is $1/\sqrt{k}$,  which matches known asymptotic rates in the literature of tail empirical processes in the univariate or multivariate case (see \emph{e.g.} \cite{einmahl1988strong} or \cite{aghbalou2021tail}, Theorem~3).  We leave to further research the question of the asymptotic behaviour of $\Covh{k} -\Cov{\tnk}$ as $k,n\to\infty$, $k/n\to 0$, a problem which could be attacked  by means of Lindeberg central limit theorems in Hilbert  spaces~(\cite{kundu2000central}).
\end{remark}
Combining  Theorems~\ref{thm:propproba} and~\ref{thm:stats}, the following consistency result is immediate. 
\begin{corollary}[Consistency]
The empirical covariance of extreme angles $\Covh{k}$ is consistent, \ie as   $n,k\to\infty$ with  $k/n\to 0$,  $$\|\Covh{k} -  \Cov{\infty}\|_{HS(\Hilb)} \to 0  \text{  in probability.  } $$
\end{corollary}

Theorem~\ref{thm:stats} also provides a control of the deviations of the empirical eigenspaces, 
with a proof paralleling the one of Corollary~\ref{cor:CCLproba}. In the following statement we denote by $\ESh{p}{k}$ such an eigenspace, that is, the linear space generated by the first $p$ eigen functions of $\Covh{k}$. 
\begin{corollary}[Deviations of empirical eigenspaces]\label{cor:CCLstats} Let
  $p\in\nset^*$ satisfying the same positive eigen gap assumption as in Corollary~\ref{cor:CCLproba}, that is $\eg{p}{\infty} := (\ev{p}{\infty} - \ev{p+1}{\infty} )/2 >0$. Denote the pre-asymptotic eigen gap by
  $$
\eg{p}{t} =\frac{\ev{p}{t}-\ev{p+1}{t}}{2}. 
$$
Let  $n,k$ be large enough so that $\eg{p}{\tnk}>0$  (see Remark~\ref{rem:choice_p_cvEigenVals} for the fact that $\eg{p}{\tnk} \to \eg{p}{\infty} >0$).  
 For  $\delta \in (0,1)$,  with probability larger than $1-\delta$, we have
\begin{align*}
\rho(\ESh{p}{k},\ES{p}{\tnk})&\le \frac{B(n,k,\delta)}{\eg{p}{\tnk}},  
\end{align*}
where $B(n,k,\delta)$ is the upper bound on the deviations of $\Covh{k}$ stated in Theorem~\ref{thm:stats}. 
In particular, we have the following consistency result as $n,k\to\infty$ while $k/n\to 0$, 
\begin{equation*}
\rho(\ESh{p}{k},\ES{p}{\infty})\to \; 0 \text{ in probability}.  
\end{equation*}

\end{corollary}

\section{Illustrative Numerical Experiments}\label{sec:num}
Two possible applications
of PCA for functional extremes are considered here. 
In both contexts, our goal  is to
assess the usefulness of the proposed functional PCA method for extremes
by comparing it with the closest alternative, namely
functional PCA of  the full sample (not only  extremes).
On the one hand, a typical objective is to identify
likely profiles of extreme events, by which we mean a finite
dimensional subspace of $\Hilb$ with basis given by the eigenfunctions
of $\Cov{\infty}$ with the highest eigenvalue. In this context,
extreme functional PCA serves as a pattern identification tool for a qualitative
interpretation. This line of thoughts is illustrated in
Section~\ref{sec:expe_sparseIdentification} on a toy simulated dataset in the multiplicative model of Example~\ref{prop:example1}.

On the other hand, functional PCA of extremes  may be viewed as a data compression tool allowing to represent functional extremes in a finite dimensional manner, with optimal reconstruction properties which would not be achieved by standard functional PCA. The relevance of this approach is demonstrated in Section~\ref{sec:expe_optimalReconstruct} with
an electricity demand dataset  which is publicly available on the \texttt{CRAN} network. On this occasion we also propose visual diagnostics for functional regular variation according to finite-dimensional characterizations proposed in Section~\ref{sec:HT}.

The electricity demand dataset \texttt{thursdaydemand} considered in Section~\ref{sec:expe_optimalReconstruct} is  
available in the \texttt{R} package \texttt{fds}.  
It contains half-hourly electricity demands on thursdays  in Adelaide between 6/7/1997 and 31/3/2007. It is made of $ n =  508$ observations $X_i$, 
each of them being represented as a vector of size
$48$, indicating  the recorded half-hour demand on day $i$.
 Here an `angle' is
in practice  the profile of the half-hour records over one day, \ie the
original curve rescaled by its $L^2$-norm.

In our toy example (Section~\ref{sec:expe_sparseIdentification}) we
generate a functional regularly varying dataset of same dimension
$d=48$ with larger sample size $n= 10e+3$, according to
Example~\ref{prop:example1}. With the notations of the latter example,
we choose $Z\in \rset^4$ with independent components, with $Z_1
\sim \Pareto{0.5}$, $Z_2 \sim 0.8 * \Pareto{0.5}$, $Z_3 \sim
\calN(m=0, \sqrt{\sigma^2} = 20)$, $Z_4 \sim \calN(m=0,
\sqrt{\sigma^2} = 0.8 * 20)$, $Z_5 \sim \calN(m=0, \sqrt{\sigma^2} =
0.6 * 20)$, $Z_4 \sim \calN(m=0, \sqrt{\sigma^2} = 0.4 * 20)$, where $\calN(m, \sqrt{\sigma^2}$ is the normal distribution with mean $m$ and variance $\sigma^2$. 
The  first two components have a heavier tail than the last four, which may be considered at noise above sufficiently high level.  The angular measure on the sphere of $\rset^4$ is concentrated on the canonical basis vectors $(e_1,e_2)$.

The $\Lspace$ functions  $A_j$'s are chosen deterministically   for simplicity, namely $A_j(x) =  \sin(2\pi\omega_j x),\,  j\in \{1,3,5\} $ and 
$A_j(x) =  \cos(2\pi\omega_j x), j\in \{2,4,5\} $, with $(\omega_1,\ldots, \omega_6) = ( 2,3,1,4,5,6)$. In this setting the angular measure of extremes in $\Lspace$ is concentrated on a two-dimensional subspace,  namely the one generated by $(A_1,A_2)$.

From a numerical perspective, all scalar products in $\Lspace$ are
approximated in this work by the Euclidean scalar product in
$\rset^{48}$, which corresponds to a Riemann midpoint rule. For
simplicity, and because the choice of the unit scale is also
arbitrary, we dispense with standardizing by the half-hour width
between records. Several numerical solutions exist to perform the
eigendecomposition of the empirical covariance operator. However the
considered datasets  are moderately high dimensional and because all
observations are regularly sampled in time we may use the simplest strategy, 
which is to perform the eigendecomposition of second
moments matrix $\mb{ X^\top X} \in \rset^{48 \times 48}$ where
${\mb X_{i,j}}$ is the $j^{th}$ time record on the $i^{th}$ day. In
practice we  rely on the \texttt{svd} function in \texttt{R}
issuing the  singular value decomposition of $X$ based on a
LAPACK routine.  This boils down to choosing as a basis for
$\Lspace$ a family of indicator functions centered at the
obervation times.
Alternative orthonormal families in $\Lspace$ (typically, the Fourier basis or wavelet basis) may be preferred in higher dimensional  contexts or with irregularly sampled observations.

\subsection{Pattern Identification of functional extremes}
\label{sec:expe_sparseIdentification}

With the synthetic dataset described above, we compare the output of
functional PCA applied to extreme angular data, to the one obtained
using all possible angles, \ie 
the eigen decomposition of $\Covh{k}$ with that of $\Covh{n}$. The
scree-plot (\ie the graph of ordered eigen values, normalized by their sum) for both operators
is displayed in Figure~\ref{fig:screeplots_eigenVects_simu.pdf}. The gap between the first two eigenvalues
and the remaining ones is more pronounced with $\Covh{k}$ than with
$\Covh{n}$, indicating that the method we promote is able to uncover  a sparsity pattern at extreme levels. The limit  measure of
extremes is indeed concentrated on a two-dimensional subspace, as opposed to the distribution of the full dataset which support has dimension four. In addition the 'true' extreme angular pattern, which is a superposition of two  periodic signals with frequencies $(1,7)$, is easily recognized on the first two eigenfunctions of the  extreme covariance $\Covh{k}$ (solid lines, first two panels of the second row in Figure~\ref{fig:screeplots_eigenVects_simu.pdf}) while these frequencies are perturbed by shorter tailed `noise' with  the full covariance $\Covh{n}$ (dotted lines). The discrepancy between extreme and non-extreme eigen functions vanishes for the third eigen function, which may be considered as `noise' as far as extremes are concerned.  
\begin{figure}[hbtb]
  \centering
  \includegraphics[width=0.75\linewidth]{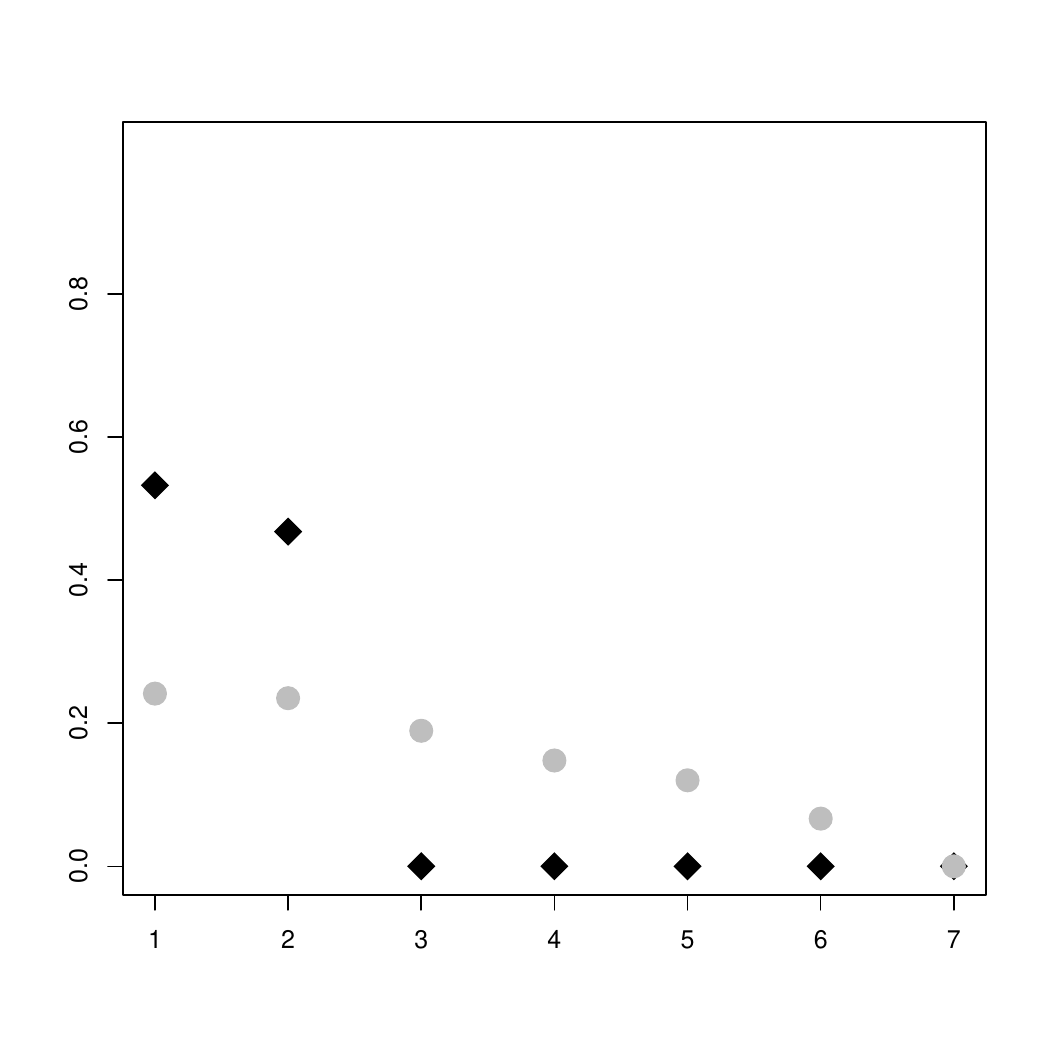}
    \includegraphics[width=0.75\linewidth]{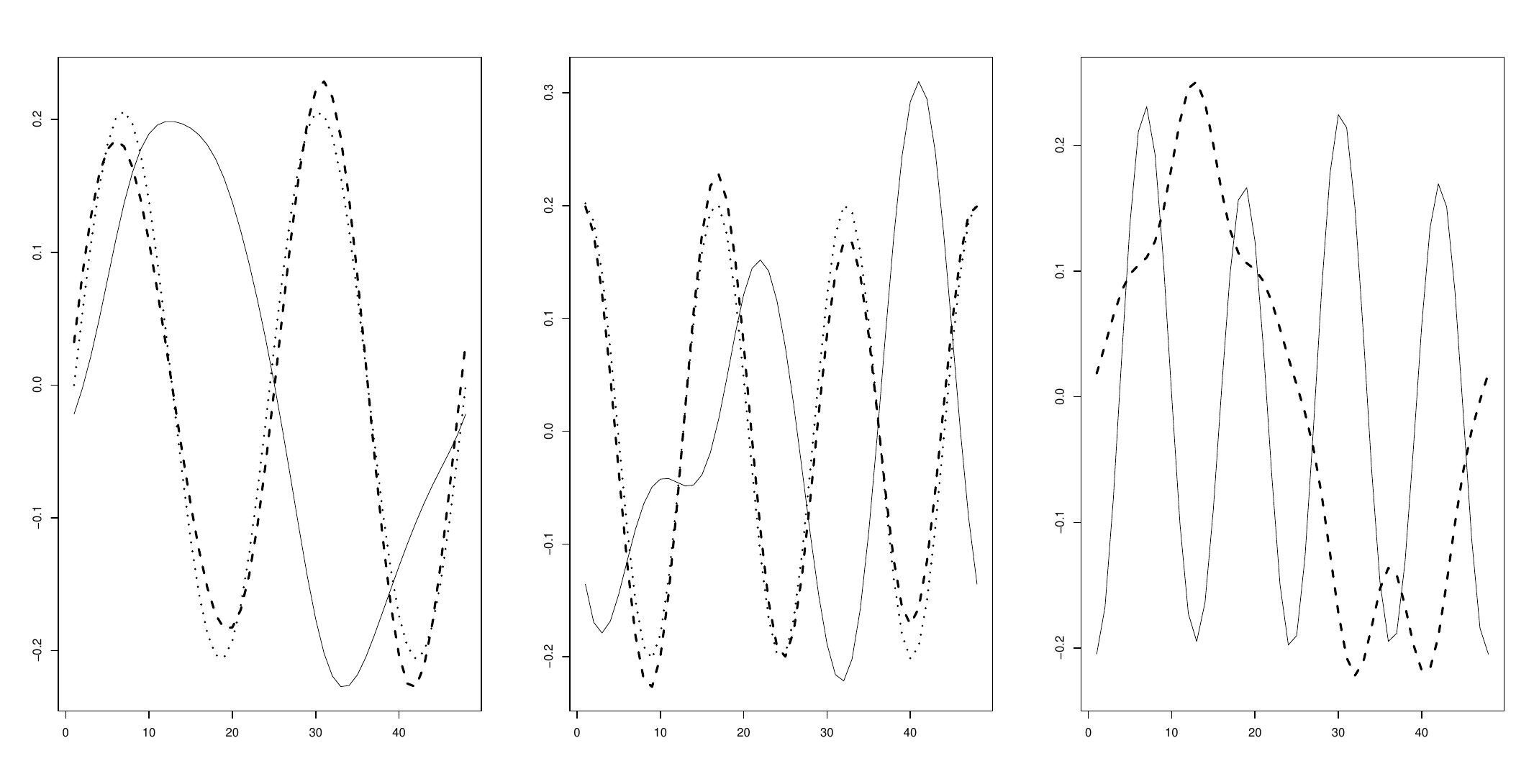}
  \caption{Simulated data: Scree plots and first three  eigenfunctions. Diamond shaped dots and  dashed lines: angular functional PCA of extremes ($\Covh{k}$). 
  Round dots and solid lines: angular functional PCA of the full dataset ($\Covh{n}$). Dotted lines on the first two plots, bottom left: (normalized) functions $A_1$, $A_2$, \ie support of the angular measure for extremes. }
  \label{fig:screeplots_eigenVects_simu.pdf}
\end{figure}

\subsection{Optimal reconstruction of functional extremes on the electricity demand dataset}
\label{sec:expe_optimalReconstruct}
 Here  we investigate the  $L^2$ reconstruction error when projecting new (test)  angular observations on the eigenspaces issued from the spectral decomposition  of the empirical covariance operator $\Covh{k}$. 
Another important  goal of this section is to provide guidelines and graphical diagnostic tools allowing to check whether functional regular variation in $L^2$ may  reasonably be assumed  for  a given  functional dataset. 
We choose to consider  the component-wise square root of the records so that the 
(squared) $L^2$ norm of each vector $X_i$ is an approximation of  the integrated
demand over a full day, 
which seems meaningful from an industrial perspective. 
For simplicity we ignore in this illustrative study any temporal
dependence from week to week.

As a first step, regular  variation must be checked and  an appropriate number $k$ of extreme observations should be selected  for estimating $\Cov{\infty}$ with $\Covh{k}$. A Gaussian QQ-plot (not shown) suggests that the radial quantile is potentially heavy-tailed. 
In view of Theorem~\ref{prop:RVfidi}, $2.$, one should check  regular variation of the radial variable and weak convergence of  univariate projections $\langle\Theta_t, h\rangle$. Regarding the radial variable $R= \|X\|$, we propose to inspect  a Hill plot and a Pareto quantile plot (\cite{beirlant2006statistics}, Chapter 2).
Visual inspection   (Figure~\ref{fig:hillAndPAretoPlots_real})
suggests a stability region for the Hill estimator
of $\gamma = 1/\alpha$ (left panel) between $k=50$ and $k=200$.
Choosing $k=100$ corresponds to an empirical quantile level
$1-k/n \approx 0.7$, for which the Pareto quantile plot (right panel)
is reasonably linear. For $k=100$ the estimated regular variation
index with the Hill estimator $\hat \gamma$ is $\hat \alpha = 1/\hat \gamma = 22.5$
($0.95$ CI: $[18.8 - 27.9]$).
\begin{figure}[h]
  \centering
    \includegraphics[width=0.35\linewidth]{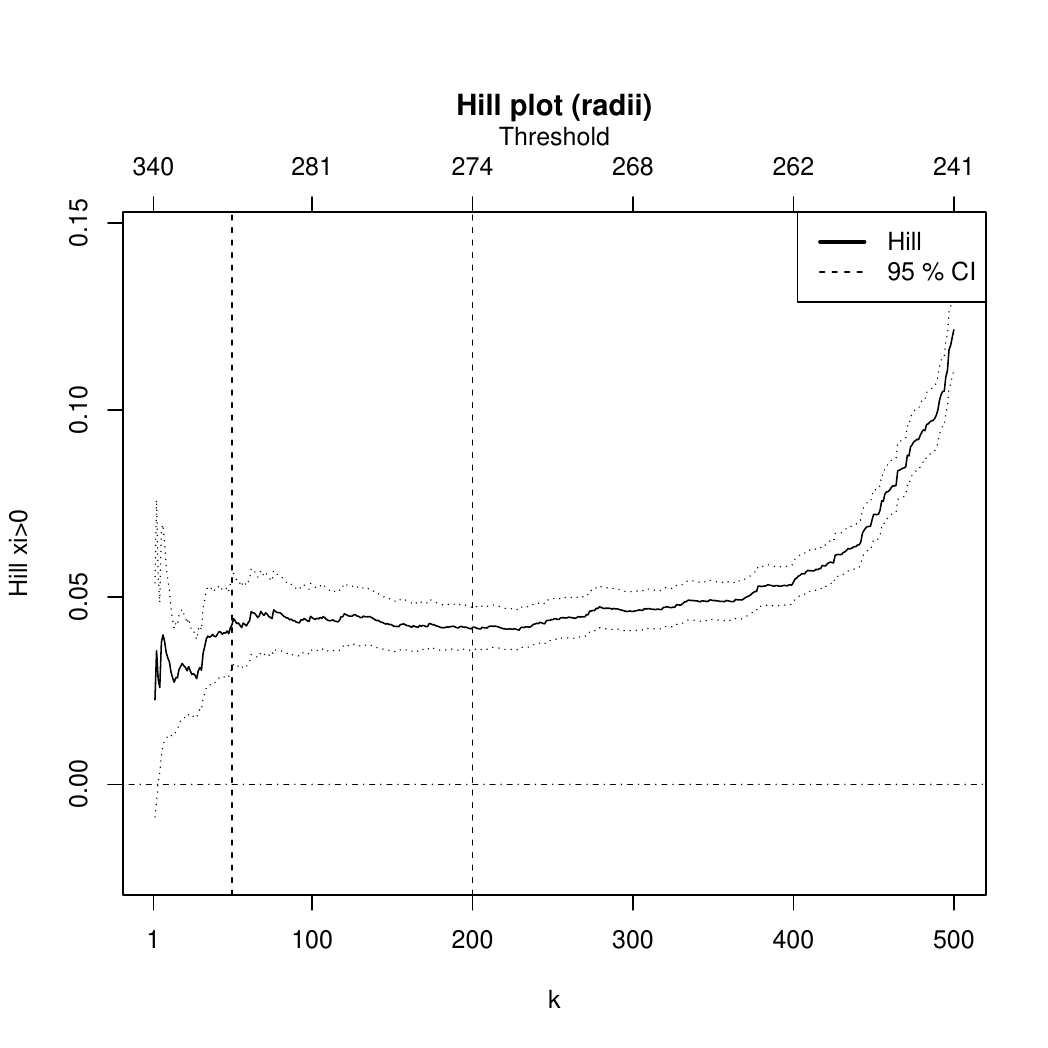}
  \includegraphics[width=0.35\linewidth]{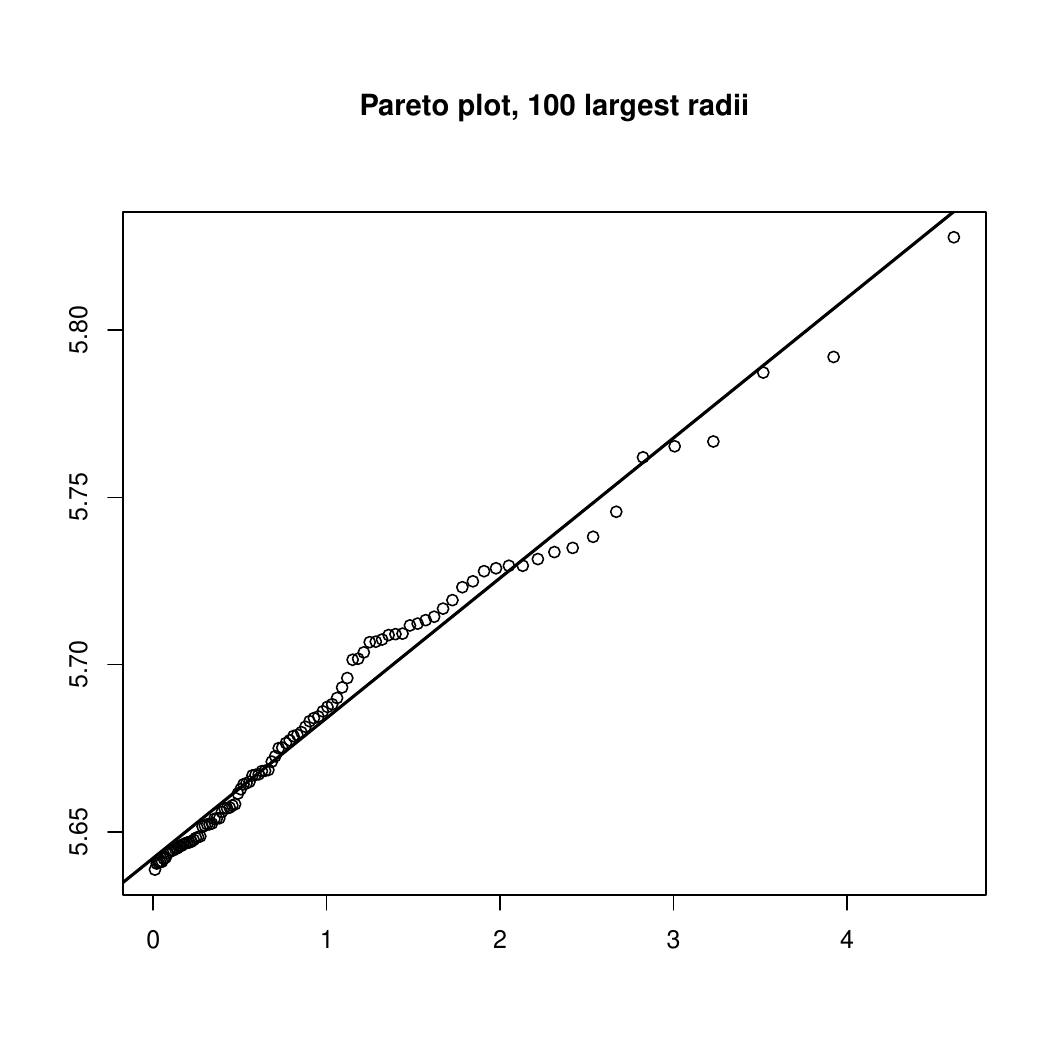}
  \caption{Hill plot (left) and Pareto quantile plot (right) for the radial variable of the air quality dataset. Dotted vertical lines on the Hill plot: stability region.}\label{fig:hillAndPAretoPlots_real}
\end{figure}

The condition of weak convergence of projections $\langle\Theta_t, h\rangle$ is obviously difficult to check in practice, in particular because it must hold for any $h$. As a default strategy we propose to check convergence of the (absolute value of) the first moment,  namely convergence of $\E | \langle \Theta_t,h\rangle |$ as $t\to\infty$,  for a finite number of `appropriate' functions $h_j, j\in\{1,\ldots, J\}$. 
The context of daily records suggests a periodic family, namely we choose $h_j (x) = \sin(2\pi j x) $, for $j \in \{1, 2, 3,4,6,8\}$.  Figure~\ref{fig:weakCvAngle} display the six plots of empirical conditional moment $\frac{1}{k}\sum_{i=1}^k | \langle \Theta_{(i)}, h_j \rangle |$. The plots confirm the existence of a relative stability region around $k=100$.
\begin{figure}[hbtp]
  \centering
  \includegraphics[width=\linewidth]{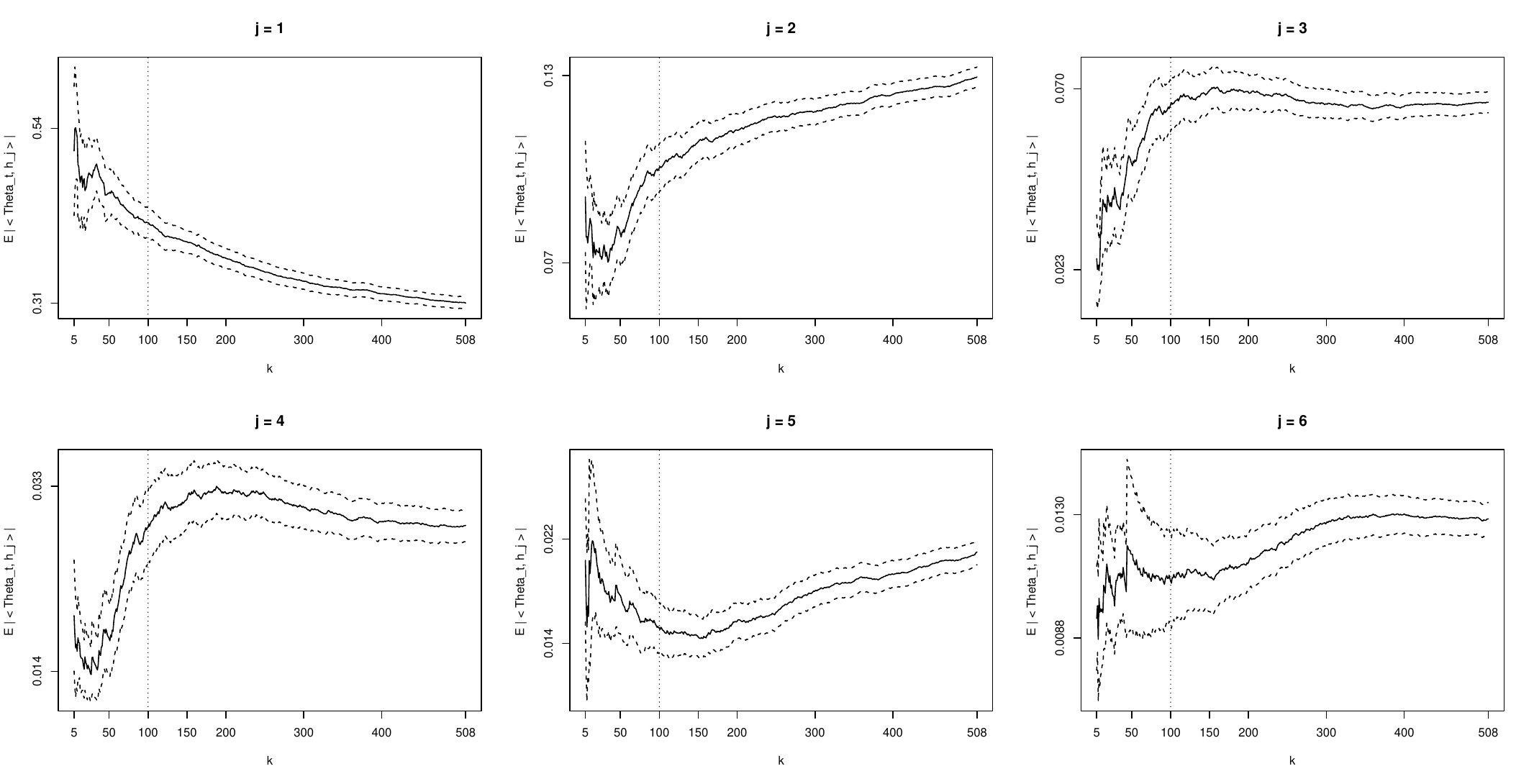}
  \caption{Air quality data: first moment of $|\langle \Theta ,  h_j\rangle|$ conditioned upon $R\ge R_{(k)}$, as a function of $k$, for $h_j(x) = \sin(2\pi j x )$, $x\in [0,1]$ }
  \label{fig:weakCvAngle}
\end{figure}

Turning to performance assessment, we consider the reconstruction squared error of a validation  subsample of  extreme angles $\mathcal{V}\subset\{\Theta_{(1)}\ldots, \Theta_{(k)} \}$ ($k=100$),   after projection on the principal eigenspaces of dimension $p$ corresponding to  three variants of the empirical uncentered angular covariance operator. In this experiment we choose $p = 2$. Namely we consider uncentered covariances   $(i)$ $\widetilde{C}_{k}$, built from  an extreme training set $\mathcal{T} = \{\Theta_{(1)}\ldots, \Theta_{(k)} \}\setminus \mathcal{V}$  ; $(ii)$ $\widetilde{C}_{n}$, incorporating  all angles (including non-extreme ones) except from the validation set,  $\{\Theta_1,\ldots, \Theta_n \}\setminus \mathcal{V}$;
$(iii)$  $\widetilde{C}_{n,k}$, built from a subsample of $\{\Theta_1,\ldots, \Theta_n \}\setminus \mathcal{V}$ of same size as $\mathcal{T}$.  
The left panel of Figure~\ref{fig:compareError} displays the boxplots of the  cross-validation error obtained over $300$ independent experiments where a validation set $\mathcal{V}$ of size $30$ is randomly chosen among $\{\Theta_{(1)}\ldots, \Theta_{(k)} \}$. The right panel displays the  out-of-sample error over a tail region, namely the validation set $\mathcal{V}$ is composed of the most extreme data $\{\Theta_{(1)}\ldots, \Theta_{(30)} \}$, and the boxplots represent the variability of the reconstruction error over the validation set. The conclusion is the same for both panels, performing functional PCA on the fraction of the angular data corresponding to the most extreme angles significantly reduces the reconstruction error, despite the reduced size of the training set. Comparison between the second and the third boxplot of each panel illustrates the negative impact of the reducing   the training sample size, while comparing the first and the third boxplots shows the bias reduction achieved by localizing on the tail region. On this particular example the bias-variance trade-off favors  our approach.

\begin{figure}[h]
  \centering
  \includegraphics[width=0.35\linewidth]{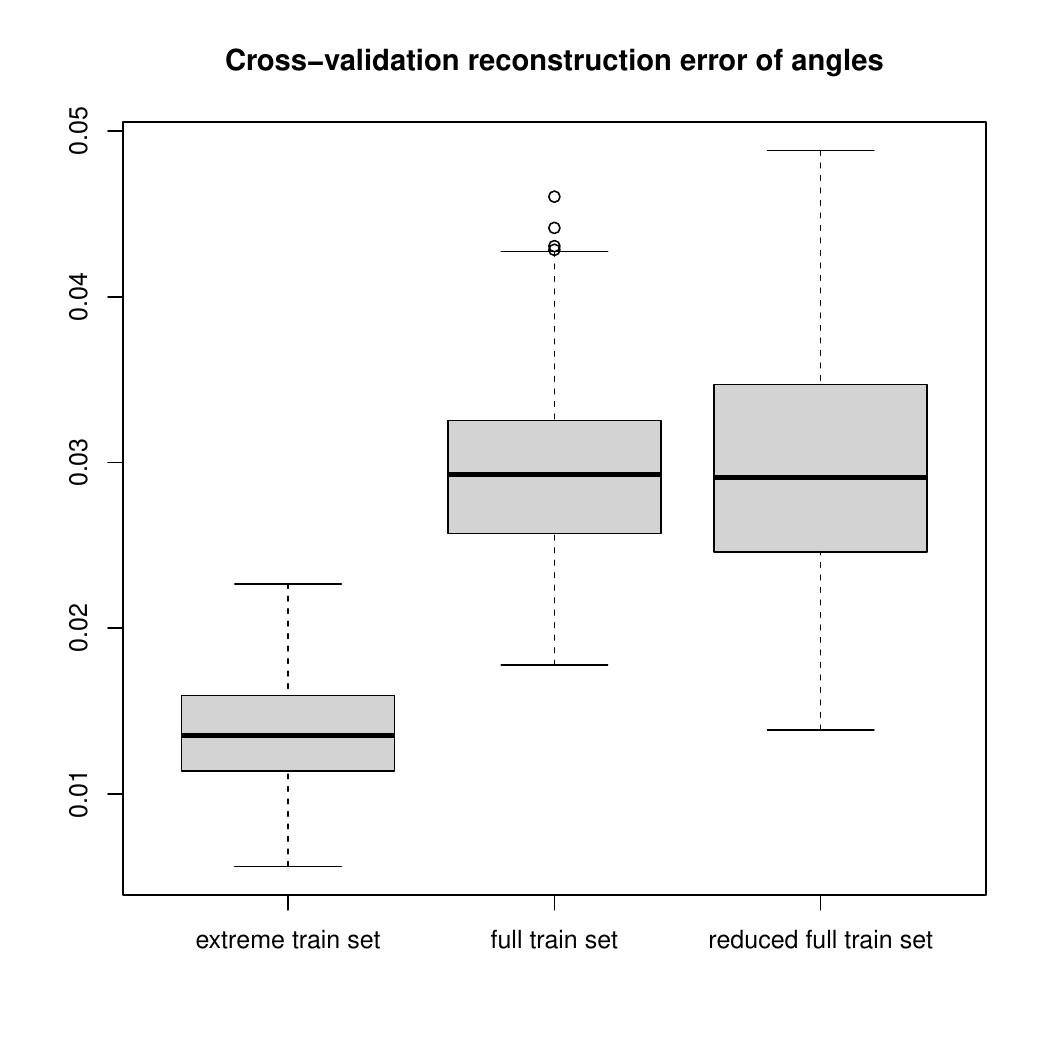}
    \includegraphics[width=0.35\linewidth]{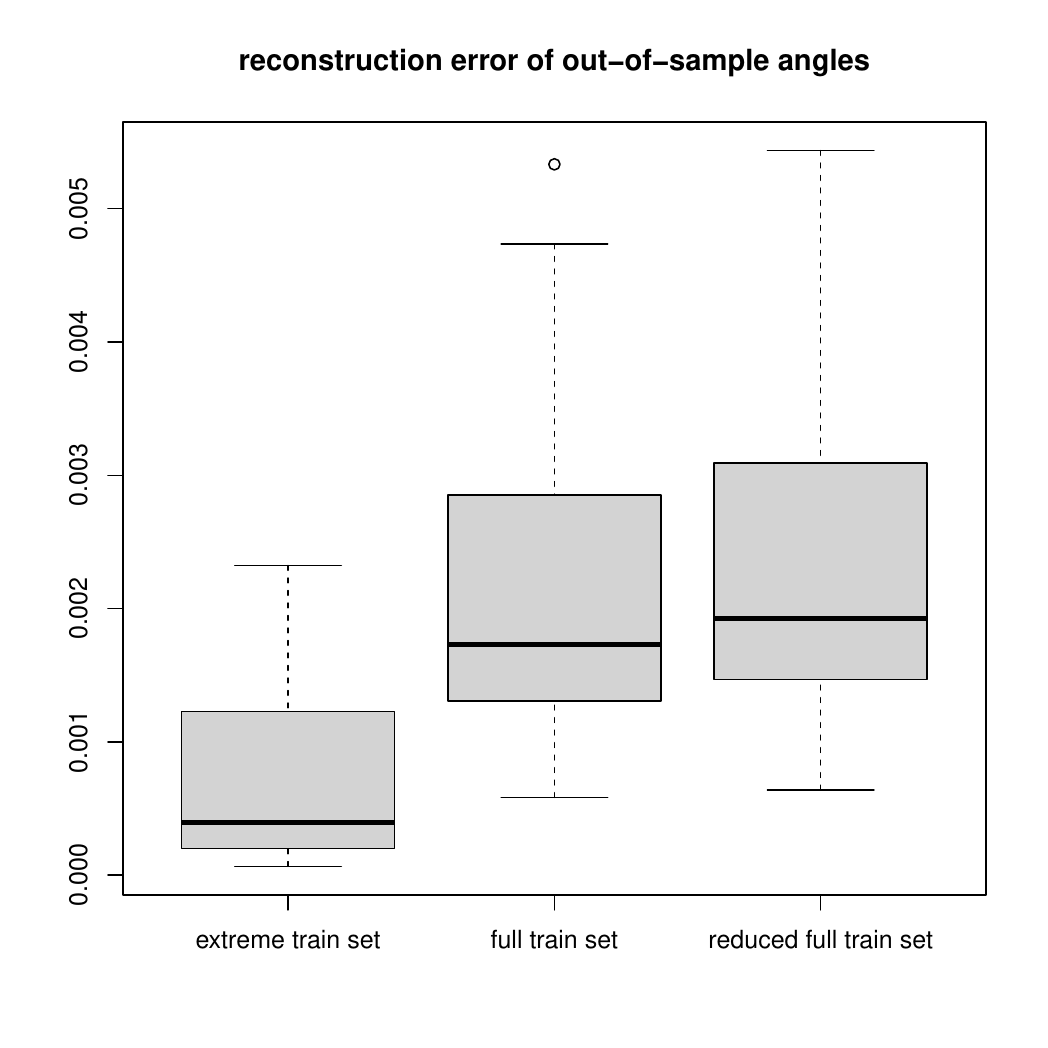}
  \caption{Cross-validation and Extrapolation error of extreme and non-extreme angular functional PCA}
  \label{fig:compareError}
\end{figure}

\clearpage
\bibliographystyle{apalike}
\bibliography{biblio_cov_ope_conv}

\clearpage
\appendix
\begin{center}
  {\huge\bf  \textsc{Appendix}}
\end{center}
\section{Proofs for Section \ref{sec:HT}}\label{sec:proofsHilbert}

\subsection*{Proof of Proposition \ref{prop:CE}.}
Consider as in the sketch of the proof the random element  $X:= R \Theta$ valued in $\Hilb$, with radial component $R=\| X\| \sim Pareto(\alpha)$ on $[1, +\infty [$, \emph{i.e.} $\forall t \s 1, \PP[R \ge t] = t^{-\alpha}$ and with  angular component $\Theta=X/\| X\|$ defined through its conditional distribution 
\begin{equation*}
	\mathcal{L}(\Theta | R ) = \frac{1}{\sum_{k=1}^{\floor{R}} 1/k}\sum_{j=1}^{\floor{R}} \frac{1}{j} \delta_{e_j}.
\end{equation*}
Notice that $r \mapsto \sum_{i=1}^{\floor{r}} \frac{1}{i}$ is slowly varying since we have
$ \sum_{i=1}^{\floor{r}} 1/i \sim  \log r $ as $r\to\infty$. 
We now check that  $X$ satisfies the properties listed in the statement.
That $\|X\| = R \in RV_{-\alpha}$ (Condition~\ref{prop:CE:1}) is obvious.  
Fix  $N \geq 1$, recall that $\pi_N(X)$ is the $\rset^N$-valued projection onto the first $N$ elements of the basis $(e_i)_i$ and denote by $\Theta_{N} = \pi_N(X)/\|\pi_N(X)\|$ the associated angular component in $\rset^N$. Denote also by $R_{N} = \| \Pi_N (X)\|_{\rset^N}$ the  radial components of $\pi_N(X)$ in $\rset^N$. 
First, we show that $R_{N} \in RV_{-\alpha}$. Observe that for all $t \ge N$,
\begin{align*}
  \PP[R_{N} \geq  t]
  &= \PP[R_{N} \ge t, R \s t]\\
  &=\PP[ \Theta \in \{ e_1,...,e_N \} , R \ge t]
    = \E \left[ \1\{R \s t \} \bb{P}\{\Theta \in \{ e_1,...,e_N \} \mid R\} \right]  \\
  &= \E \left[ \1\{R \s t \} \frac{\sum_{i=1}^{N}1/i}{\sum_{l=1}^{\floor{R}}1/l} \right] 
    = \sum_{i=1}^{N}\frac{1}{i}\int_t^{+\infty}  \frac{\alpha r^{-(\alpha+1)}}{\sum_{l=1}^{\floor{r}}1/l} \ud r .
\end{align*}
Since
$r \mapsto 
\alpha
r^{-(\alpha+1)}/(\sum_{l=1}^{\floor{r}}1/l) \in RV_{- (\alpha +1 )}$,
we have $R_{N} \in RV_{- \alpha}$ by virtue of Karamata's theorem, see
\textit{e.g.} p.25 in \cite{resnick2007heavy}.

We next prove that $\mathcal{L}(\Theta_{N} | R_{N} \s t)$ weakly
converges as $t\rightarrow +\infty$. First, since for all $t$, the
measure $\PP[\Theta_N\in \point \given R_n>t]$ is supported by the
finite set $\{e_1,...,e_N\}$. It is sufficient to show convergence of
each $\PP[\Theta_N = e_j \given R_N \ge t]$ towards some $p_j\ge 0$
for all $j$, with $\sum_{j\le N} p_j = 1$. For fixed $j\le N$ and for $t \ge N$,
	\begin{align}
          \PP[\Theta_{N} = e_j \given R_{N} \ge t ]
          &= \frac{\PP[\Theta = e_j , R_{N} \ge t ]}{\PP[R_{N} \ge t ] } \nonumber \\
           &= \frac{\PP[\Theta = e_j , R \s t ] }{\PP[R_{N} \ge t ] } \nonumber \\
          & = \frac{\PP[\Theta = e_j, R\ge t]}{\PP[\Theta \in \{e_1,\ldots, e_N\}, R\ge t]} \nonumber \\
         &= \frac{\E\bigg[\1\{R \ge t \} \frac{1/j}{
           \sum_{l=1}^{\floor{R}}1/l}\bigg]}{\E\bigg[\1\{R \s t \} \frac{\sum_{i=1}^N 1/i}{\sum_{l=1}^{\floor{R}}1/l}\bigg]} \nonumber \\
          &= \frac{1/j}{\sum_{l=1}^{N}1/l} \nonumber
	\end{align}
Hence, when $t$ is large enough, $\mathcal{L}(\Theta_{N} | R_{N} \s t) = \frac{1}{\sum_{l=1}^{N}1/l} \sum_{i=1}^N \frac{1}{i} \delta_{e_i}$.
We have shown that  for all $N \in \bb{N}$, $\pi_N(X)$ is regularly varying in $\rset^N$ with tail index $-\alpha$.

	\medskip
	
We now  show that $X \notin \RV(\bb{H})$. Since $\|X\| \in \RV_{-\alpha}(\bb{R})$, according to Proposition~\ref{prop:polarRVCond}, we have to prove that $\mathcal{L}\Big(\Theta \,|\, R\s t\Big)$ does not converge when $t$ tends to infinity. From  Theorem 1.8.4. in~\cite{vaart1997weak}, it is enough to show that the sequence of measures $P_{\Theta,n}  = \PP[\Theta \in \point \given  R> n]$  is not asymptotically finite-dimensional.  
\emph{i.e.} that 
	 \begin{equation*}
	 \exists \delta,\varepsilon >0, \forall d \in \bb{N}^*, \limsup_n \PP[ \sum_{i> d } \langle \Theta, e_i \rangle^2 \s \delta \bigg| R \s n ] \s \varepsilon. 
	 \end{equation*}
Let $\delta > 0, \varepsilon \in ]0,1[$ and let $n>d$.
\begin{align}
  \PP[ \sum_{i> d } \langle \Theta, e_i \rangle^2 \s \delta \; \bigg| \;  R \s n ]
  &= \frac{\PP[ \Theta \notin \{ e_1,...,e_d\} , R \s n ]}{\bb{P}\{R \s  n\}} \nonumber \\
  &= \frac{\E\bigg[\1\{R \s n\} \PP[ \Theta \notin \{ e_1,...,e_d\} \;|\; R] \bigg]}{\PP[R \s n]} \nonumber \\
  &= \frac{\E\bigg[\1\{R \s n\}\big(1-\frac{\sum_{i=1}^d 1/i}{\sum_{l=1}^{\floor{R}} 1/l}\big) \bigg]}{\PP[R \s n]} \nonumber \\
  &= \E \bigg[ 1 - \frac{\sum_{i=1}^d 1/i}{\sum_{l=1}^{\floor{R}}1/l} \; \bigg|\;  R \s n \bigg] \nonumber \\
  &\s 1 - \frac{\sum_{i=1}^d 1/i}{\sum_{l=1}^{n}1/l} \xrightarrow[n\to\infty]{} 1 \s \varepsilon \label{eq:notfidi} 
\end{align}
Hence, the asymptotic finite-dimensional condition does not hold and $X$ is not regularly varying in $\bb{H}$.

\subsection*{Proof of the claim in Example~\ref{prop:ex3}.}
We show that $X$ is regularly varying in $\Hilb$. 
Following the lines of the proof of Proposition \ref{prop:CE}, it is enough to verify  that $\Pangt \wto \Theta_\infty$. Since the common support of $\Pangt$  and $\Panginf$ is discrete we only need to show that
 $\PP[\Theta  = e_j \given R>t]\to 6/( \pi j)^2$ for fixed $j \in\nset^*$.

For such $j$, following the steps leading to \eqref{eq:notfidi}, we find
\begin{align*}
  \PP[\Theta = e_j \given R\ge t]
  & =\frac{ \EE\Big[ \1\{ R \ge t \} \PP[\Theta = e_j \given R] \Big]}{ \PP[R\ge t] }\\
   & = \frac{ \EE[ \1\{ R \ge t \} j^{-2} /\sum_{l = 1}^{\lfloor R\rfloor } l^{-2} ]}{ \PP[R\ge t] }\\
   &= j^{-2} \EE[1/\sum_{l=1}^{\lfloor R\rfloor } l^{-2}  \,\Big|\, R\ge t ] \\
   &= j^{-2} \int_t^{+\infty} (\sum_{l=1}^{\floor{r}}l^{-2} )^{-1} \frac{\alpha t^\alpha}{r^{\alpha+1}} \ud r
\end{align*}

The integrand in the latter display is a regularly varying function of $r$ with exponent $-\alpha - 1$, and  an application of Karamata's theorem yields that
$\PP[\Theta = e_j \given R\ge t] \to 6/(\pi j^2)$.

\section{Proofs for Section~\ref{sec:KL}}\label{sec:proofsPCA}

\subsection*{Proof of Proposition~\ref{prop:prop4.1}.}

Our main tool to derive a  concentration bound on $\|\Covb{\tnk} - \Cov{\tnk}\|_{HS(\Hilb)}$ is  a Bernstein-type inequality, Theorem~3.8 in \cite{mcdiarmid1998concentration} 
which we recall for convenience.  Here and throughout we adopt the shorthand notation $x_{i:j} = x_i, \ldots, x_j$ for $i\le j$.
\begin{lemma}[Bernstein's type inequality] \label{lem:Bernstein}
Let $\textbf{X}=(X_{1:n})$ with $X_i$ taking their values in a set $\mathcal{Z}$ and let $f$ be a real-valued function defined on $\mathcal{Z}^n$. Let $Z = f(X_{1:n})$. Consider the positive deviation functions, defined for $1 \le i \le n$ and for $x_{1:i} \in \mathcal{Z}^i$
\begin{equation*}
g_i(x_{1:i}) = \E [Z | X_{1:i}=x_{1:i}] - \E [Z | X_{1:i-1}=x_{1:i-1}].
\end{equation*}
Denote by $b$ the maximum deviation 
\begin{equation*}
b = \max_{1 \m i \m n} \sup_{x_{1:i} \in \mathcal{Z}^i} g_i(x_{1:i}).
\end{equation*}
Let $\hat{v}$ be the supremum of the sum of conditional variances, 
\begin{equation*}
\hat{v} := \sup_{(x_1,...,x_n) \in \mathcal{Z}^n} \sum_{i=1}^n \sigma_i^2(f(x_1,...,x_n)),
\end{equation*}
where $\sigma_i^2(f(x_{1:n})):= \Var[g_i(x_{1:i-1},X_i)]$.
If $b$ and $\hat{v}$ are both finite, then 
\begin{equation*}
\bb{P}(f(\textbf{X}) - \EE[f(\textbf{X})] \s \varepsilon ) \m \exp\bigg(\frac{-\varepsilon^2}{2(\hat{v}+b\varepsilon/3)}\bigg),
\end{equation*}
for $u \s 0$.
\end{lemma}

In order to apply this inequality to our purposes we need to write the empirical pre-asymptotic operator (or its surrogate $\Covb{\tnk}$) as a function $f_t$ of the sample $X_{1:n}$. With this in mind,  we introduce a thresholded angular functional
\begin{equation*}
 \fonction{\theta_t}{ \Hilb}{\sphere}{x}{\theta_t(x) = \1\{\|x\|\ge t\}\|x\|^{-1} x\;.}      
\end{equation*}
Observe that with this notation, $\Theta_i\1\{R_i>t\}=\theta_t(X_i)$. Consider now the function 
\begin{equation*}
  \fonction{f_t}{\Hilb^n}{\rset}{x_{1:n}}{
    f_{\tnk}(x_{1:n}) \;=\; \frac{1}{k}\| \sum_{i=1}^n (\theta_t(x_i) \otimes \theta_t(x_i) - \E [\theta_t(X) \otimes \theta_t(X)] \|_{HS(\Hilb)} \;.} 
\end{equation*}
Notice that $f_{\tnk}(X_{1:n})=\|\Covb{\tnk} - \Cov{\tnk}\|_{HS(\Hilb)}$ which is the focus of Proposition~\ref{prop:prop4.1}.

\begin{lemma}[deviations of $f_{\tnk}(X_{1:n})$]\label{lem:deviationsFtnk}
With the above notations, we have
\begin{equation*}
\PP(f_{\tnk}(X_{1:n}) - \EE [f_{\tnk}(X_{1:n})] \ge \varepsilon ) \le \exp \bigg( \frac{-k\varepsilon^2}{4(1+\frac{\varepsilon}{3})}\bigg).
\end{equation*}

\end{lemma}
\begin{proof}
We apply  Lemma~\ref{lem:Bernstein} to the function $f=f_{\tnk}$. To do so we derive upper bounds on the maximum deviation term $b$ and on the maximum sum of variances $\sigma^2$ from the statement.  Let $x_{1:n}\in \Hilb^n$.  The maximum deviation $b$ is bounded by $2/k$ since by independence among $X_i$'s, with the notations of Lemma~\ref{lem:Bernstein}, 
\begin{align*}
g_i(x_{1:i})&=  \EE[ f_{\tnk}(x_1,...,x_{i-1},x_i,X_{i+1},...,X_n) -  f_{\tnk}(x_1,...,x_{i-1},X_i,X_{i+1},...,X_n)]   \\ 
&\le \frac{1}{k} \EE[\| \theta_{\tnk}(x_i) \otimes \theta_{\tnk}(x_i)- \theta_{\tnk}(X_i) \otimes \theta_{\tnk}(X_i) \|_{HS(\Hilb)}]  \\
&\m \frac{1}{k}(\1 \{ \|x_i \| \s \tnk \} + \PP[\|X\| \s \tnk ]) \m \frac{1+k/n}{k} \m \frac{2}{k}, 
\end{align*}
where the first inequality comes from the triangle inequality
$|\|a\| - \|b\| |\le \|a-b\|$, and the second one from the fact that $\|s\otimes s\|_{HS(\Hilb)} = 1$ if $\|s\| = 1$.

There remains to bound the variance term. Since for every
$1 \le i \le n$, by the tower rule for conditional expectations,
$\EE[g_i(x_{1:i-1},X_i)] = 0$, we may write, for $Y_i$ and independent copy of $X_i$,  
\begin{align*}
  \sigma_i^2(f_{\tnk}(x_1,...,x_n))
  &= \EE[(f_{\tnk}(x_1,...,x_{i-1},Y_i,X_{i+1},...,X_n)-
    f_{\tnk}(x_1,...,x_{i-1},X_i,X_{i+1},...,X_n))^2]  \\
&\le \frac{1}{k^2} \EE[\| \theta_{\tnk}(Y_i) \otimes \theta_{\tnk}(Y_i) -\theta_{\tnk}(X_i) \otimes \theta_{\tnk}(X_i)\|_{HS(\Hilb)}^2]  \\
&\m \frac{2}{k^2}\EE [\|\theta_{\tnk}(X) \otimes \theta_{\tnk}(X)\|_{HS(\Hilb)}^2 ]  \\
&= \frac{2\PP[\|X\| \s \tnk]}{k^2}=\frac{2}{nk}. 
\end{align*}
Hence, $\hat{v}$ is bounded from above by $2/k$. injecting the upper bounds on $\hat v$ and $b$ in Lemma~\ref{lem:Bernstein} concludes the proof.
\end{proof}
The following intermediate lemma proves useful for bounding the expected deviation in the left-hand side of Lemma~\ref{lem:deviationsFtnk}. 
\begin{lemma}\label{lem:pythagoreInExpectation}
Let $A_1,...,A_n$ be independent centered random elements in $\Hilb$. Then 
\begin{equation*}
\EE[\bigg\|\sum_{i=1}^n A_i \bigg\|^2]=\sum_{i=1}^n \EE[\| A_i \|^2].
\end{equation*}
\end{lemma}
\begin{proof}
  The left-hand side equals  $\sum_{i=1}^n \EE [\|A_i\|^2 ] + 2 \sum_{1 \m i < l \m n } \EE[\langle A_i, A_l \rangle ]$. Since the $A_i$'s are independent  with mean $0$, for all $1 \m i < l \m n$,
  $$0 = \langle \EE  [A_i],\EE [A_l] \rangle  = \bb{E} \big[ \langle A_i, \EE [A_l] \rangle \big] = \bb{E} [ \langle A_i, \EE [ A_l | A_i] \rangle ] = \E \big[ \E [\langle A_i,  A_l \rangle | A_i] \big]=\E [ \langle A_i,A_l \rangle ],$$ which  concludes the proof.
\end{proof}
We are now ready to obtain a bound on  $\EE\|\Covb{\tnk} - \Cov{\tnk}\|_{HS(\Hilb)}$.
\begin{lemma}\label{lem:BoundExpectedDeviation}
\begin{equation}
\EE [\|\Covb{\tnk} - \Cov{\tnk}\|_{HS(\Hilb)}] \m \frac{1}{\sqrt{k}}.
\end{equation}
\end{lemma}
\begin{proof}
\begin{align*}
  \EE [\|\Covb{\tnk} - \Cov{\tnk}\|_{HS(\Hilb)}]
  &= \frac{n}{k}\bb{E} \Big[\; \Big\| \frac{1}{n}\sum_{i=1}^n \theta_{\tnk}(X_i)\otimes\theta_{\tnk}(X_i) - \EE[\theta_{\tnk}(X) \otimes \theta_{\tnk}(X)]\Big\|_{HS(\Hilb)}\;\Big]  \\
  &\m \frac{n}{k} \bb{E}\Big[ \; \Big\|
    \frac{1}{n}\sum_{i=1}^n   \theta_{\tnk}(X_i)\otimes\theta_{\tnk}(X_i)  -
    \EE[\theta_{\tnk}(X_i) \otimes \theta_{\tnk}(X_i)]  \Big\|_{HS(\Hilb)}^2
    \; \Big]^{1/2} \\
  & =  \frac{n}{k}  \frac{1}{\sqrt{n}} \bb{E} \Big[\;
    \big\|\theta_{\tnk}(X) \otimes \theta_{\tnk}(X) -
    \EE[\theta_{\tnk}(X) \otimes \theta_{\tnk}(X)] \big\|_{HS(\Hilb)}^2 \; \Big]^{1/2} \\
  &= \frac{\sqrt{n}}{k}  \Big( \bb{E}\big[\|\theta_{\tnk}(X) \otimes \theta_{\tnk}(X)\|_{HS(\Hilb)}^2\big] -
    \big\| \EE [ \theta_{\tnk}(X) \otimes \theta_{\tnk}(X)]\big\|_{HS(\Hilb)}^2
    \Big)^{1/2} \\
  &\m \frac{\sqrt{n}}{k} \EE[ \|\theta_{\tnk}(X) \otimes \theta_{\tnk}(X)\|_{HS(\Hilb)}^2 ]^{1/2} \\
  &\m \frac{\sqrt{n}}{k}\PP[\|X\| \s \tnk]^{1/2} = \frac{1}{\sqrt{k}}, 
\end{align*}
where the second identity derives from Lemma~\ref{lem:pythagoreInExpectation}, 
The last inequality follows from $\|\theta(x)\otimes\theta(x)\|_{HS(\Hilb)} = 1$.
\end{proof}

Combining Lemmata~\ref{lem:deviationsFtnk} and~\ref{lem:BoundExpectedDeviation}, together with the definition of $f_{\tnk}$, we obtain that  with probability at least $1- \gamma/2$,
\begin{equation*}
\|\Covb{\tnk} - \Cov{\tnk}\|_{HS(\Hilb)} \m \frac{1}{\sqrt{k}}+\frac{4}{3k}\log(2/\gamma)+4\bigg(\bigg(\frac{\log(2/\gamma)}{3k}\bigg)^2+\frac{\log(2/\gamma)}{k} \bigg)^{1/2}.
\end{equation*}
Simplifying the above display with $\sqrt{a+b}\le \sqrt{a}+\sqrt{b}$ yields the statement of Proposition~\ref{prop:prop4.1}.

\end{document}